\documentclass[12pt]{amsart}
\usepackage{amssymb}
\usepackage{amsthm,amsfonts}

\usepackage{amsfonts}
\usepackage{amsmath}
\usepackage[abbrev]{amsrefs}
\usepackage{enumerate}
\usepackage{color}

\usepackage[bookmarks=true,hyperindex,pdftex,colorlinks, citecolor=blue,linkcolor=blue, urlcolor=blue]{hyperref}

\newtheorem{Theorem}{Theorem}[section]

\newtheorem{Lemma}[Theorem]{Lemma}
\newtheorem{Corollary}[Theorem]{Corollary}
\newtheorem{Proposition}[Theorem]{Proposition}

\newtheorem{corollary}[Theorem]{Corollary}
\newtheorem{proposition}[Theorem]{Proposition}
\newtheorem{question}[Theorem]{Question}

\theoremstyle{definition}

\newtheorem{Definition}[Theorem]{Definition}

\newcommand{\R}{\mathbb{R}}

\newcommand{\N}{\mathbb{N}}

\newcommand{\Lip}{{\mathrm{Lip}}_0}
\newcommand{\Lin}{\mathcal{L}}
\newcommand{\eps}{\varepsilon} 
\newcommand{\slope}{\textup{slope}}
\newcommand{\conv}{\textup{conv}}
\newcommand{\re}{\textup{Re}\,}
\newcommand{\Id}{\textup{Id}}

\begin{document}

\title{On the Lipschitz numerical index of Banach spaces}

\author[Choi]{Geunsu Choi}
\address[Choi]{Department of Mathematics Education,
	Dongguk University, Seoul 04620, Republic of Korea}
\email{chlrmstn90@gmail.com}

	\author[Jung]{Mingu Jung}
	\address[Jung]{Basic Science Research Institute and Department of Mathematics, POSTECH, Pohang 790-784, Republic of Korea \newline
		\href{http://orcid.org/0000-0000-0000-0000}{ORCID: \texttt{0000-0003-2240-2855} }}
	\email{\texttt{jmingoo@postech.ac.kr}}

\author[Tag]{Hyung-Joon Tag}
\address[Tag]{Research Institute for Natural Sciences, Hanyang University, Seoul 04763, Republic of Korea}
\email{hjtag4@hanyang.ac.kr}

\thanks{The first author was supported by Basic Science Research Program through the National Research Foundation of
Korea(NRF) funded by the Ministry of Education, Science and Technology [NRF-2020R1A2C1A01010377]. The second author was supported by NRF [NRF-2019R1A2C1003857] and by POSTECH Basic Science Research Institute Grant [NRF-2021R1A6A1A10042944]. The third author was supported by Basic Science Research Program through the National Research Foundation of Korea(NRF) funded by the Ministry of Education, Science and Technology [NRF-2020R1C1C1A01010133]}

\begin{abstract}	
We provide some new consequences on the Lipschitz numerical radius and index which were introduced recently. More precisely, we give some renorming results on the Lipschitz numerical index, introduce a concept of Lipschitz numerical radius attaining functions in order to show that the denseness fails for an arbitrary Banach space, and study a Lipschitz version of Daugavet centers. Furthermore, we discuss the Lipschitz numerical index of vector-valued function spaces, absolute sums of Banach spaces, the K\"othe-Bochner spaces, and Banach spaces which contain a dense union of increasing family of one-complemented subspaces.
\end{abstract}

\maketitle

\section{Introduction and Preliminaries}

In this article, we study the Lipschitz numerical radius and index of Banach spaces and Lipschitz versions of related topics. The numerical radius and index of Banach spaces have been active research area for decades, and there have been fruitful results connected to other geometrical properties of Banach spaces. There also have been a plethora of results extending the concepts of the numerical radius and index by considering nonlinear maps such as polynomials and holomorphic functions \cite{AK, CK, CGKM, CGMM, GGMMM, FGKLM, H}. 
However, there are relatively fewer results known with respect to Lipschitz maps \cite{KMMW, W1, WHT}. So we aim to provide additional observations on the Lipschitz numerical radius, index and the related notions such as the Daugavet centers in the Lipschitz setting. 

Before providing the overview of extending the concepts to the Lipschitz settings, let us fix some notations. Let $X$ and $Y$ be Banach spaces and $X^*$ be the topological dual space of $X$. The closed unit ball and the unit sphere of $X$ are denoted by $B_X$ and $S_X$, respectively. We write by $\Lin(X,Y)$ the space of all bounded linear operators from $X$ into $Y$.

For a bounded linear operator $T \in \Lin (X, X)$, the \emph{numerical radius} $\nu(T)$ of $T$ is defined by
\begin{equation}\label{eq:classicnr}
\nu(T) = \sup \{ |x^*(Tx)| : (x,x^*) \in S_X \times S_{X^*},\, x^*(x)=1 \},
\end{equation}
and the \emph{numerical index} $n(X)$ of $X$ is given by
$$
n(X) = \inf \{\nu(T): T \in S_{\Lin(X,X)}\}.
$$
Useful information on the numerical radius and the index can be found in \cite{BD2}. We also refer to \cite{KMP} for a comprehensive summary of these topics. 

The numerical index is closely related to the \emph{alternative Daugavet property}, a property for $X$ satisfying the equation
\begin{equation}\label{eq:ADP}
\max_{|\omega| = 1}\|\Id_X + \omega T\| = 1 + \|T\|
\end{equation} 
for every rank-one operator $T \in \Lin(X, X)$. If the above equality holds for $\omega=1$, we say that $X$ has the {\it Daugavet property}. It is shown in \cite[Remark, pp. 483]{DMPW} that a Banach space $X$ has the numerical index 1 if and only if every operator $T \in \Lin(X,X)$ satisfies the equation (\ref{eq:ADP}). Hence every Banach space with the numerical index 1 has the alternative Daugavet property. Lush spaces such as $C(K)$, $L_1(\mu)$, and the uniform algebra are known to have the numerical index 1 \cite[Proposition 2.2]{BKMW}. On the other hand, $L_1(\mu)$ with a non-atomic measure $\mu$ \cite{Lo}, $C(K)$ with a perfect, compact Hausdorff space $K$ \cite{Dau}, and the uniform algebra which Shilov boundary contains no isolated point \cite{W, Wo} have the Daugavet property. Another direction of research related to the Daugavet property is considering so called Daugavet centers. Given Banach spaces $X$ and $Y$, recall that $G \in \Lin(X,Y)$ is said to be a {\it Daugavet center} if
\begin{equation*}
	\|G + T\| = \|G\| + \|T\|
\end{equation*}
for every rank-one operator $T \in \Lin(X,Y)$. Daugavet centers are introduced in \cite{BK} to consider an operator that behaves similarly to $\Id_X$ on a Banach space $X$ with the Daugavet property. For more details on Daugavet centers and their connection to recent developments in the theory of numerical index, we refer to \cite{BK, I, KMMPQ, KMMP}. 

As mentioned earlier, nonlinear versions of these concepts have been studied by many researchers. Recently, the Lipschitz numercial radius and index have been introduced in \cite{WHT, W1} via the Lipschitz numerical range, which was inspired by the numerical range of operators on nonlinear Hilbert spaces \cite{Z}. As usual, we denote by $\Lip(X,Y)$ the space of Lipschitz maps from $X$ into $Y$ that vanishes at 0, equipped with the Lipschitz norm $\|\cdot\|_L$ given by 
\[
\|T \|_L = \sup \left\{ \frac{\|Tx - Ty\|}{\|x-y\|} : x, y \in X,\, x \neq y \right\}. 
\]
For $T \in \Lip(X,X)$, the \emph{Lipschitz numerical radius} of $T$ is defined by
\begin{equation}\label{eq:Lipnr}
\nu_L(T) = \sup \left\{ \frac{|x^*(Tx-Ty)|}{\|x-y\|^2} : x^* \in D(x-y), \, x,y \in X, \, x \neq y \right\},
\end{equation}
where 
\[
D(x) = \{x^* \in X^* : x^*(x) = \|x^*\|\|x\| = \|x\|^2\}.
\]
Throughout the paper, we denote by $\Pi_X \subset X \times X \times X^*$ the \emph{state of $X$} given by 
$$
\Pi_X = \{(x,y,x^*): x, y \in X,\, x \neq y,\, x^* \in D(x-y)\}
$$
and denote by $\pi$ the projection from $X \times X \times X^*$ to $X \times X$ given by $\pi (x,y,x^*)= (x,y)$. 
We shall simply write $\Pi$ without the subscript $X$ when there is no possible confusion. 
In a similar way to the numerical index, 
the \emph{Lipschitz numerical index} $n_L(X)$ of $X$ is defined by
$$
n_L(X) = \inf \{\nu_L(T) : T \in S_{\Lip(X,X)}\}.
$$

A natural, but unsolved, question is whether the Lipschitz numerical index $n_L(X)$ is the same as the numerical index $n(X)$ or not. Clearly, we have $n_L(X) \leq n(X)$. Although the reverse inequality is not known in general, the equality turns out to be true for certain cases. For example, these notions coincide for separable Banach spaces with the Radon-Nikod\'ym property \cite{WHT}. 
Moreover, it is observed that if $X$ has the Radon-Nikod\'ym property, or $X$ is an Asplund space, or $X$ does contain $\ell_1$, or $X$ has the convex point of continuity property, then $n(X)=1$ is equivalent to $n_L (X) = 1$ \cite{KMMW}.
%Moreover, the stability results on the Lipschitz numerical index of $C(K, X)$, $L_1(\mu, X)$, $L_{\infty}(\mu, X)$, and $c_0$-, $\ell_1$-, and $\ell_{\infty}$-sums of Banach spaces are shown \cite{WHT}. 

%%%%

The article consists of three main parts. All the necessary terminologies and conventions will be presented thoroughly in each subsection because the article covers a wide range of topics. In Section \ref{subsection:renorming}, we extend the renorming results in the numerical index by showing that the Lipschitz numerical index varies continuously with respect to the equivalent renorming. In Section \ref{subsection:attaining}, we define a notion of Lipschitz numerical radius attaining maps which turns out to be related to numerical radius attaining operators. In particular, we show that the set of Lipschitz numerical radius attaining maps is not dense in $\Lip(X, X)$ for any Banach space $X$, while the G\^ateaux derivative of a Lipschitz numerical radius attaining map can be dealt as a numerical radius attaining operator. Section \ref{subsection:center}, which was inspired by \cite{KMMW}, concerns a Lipschitz version of Daugavet centers. We compare this notion with the original one and prove that these notions provided the linearity are actually equivalent to each other. 

%Plenty of Lipschitz maps beyond the weakly compact ones are introduced to satisfy the previous equation \eqref{eq:Daucent}.

In Section \ref{section:function-algebra}, we present the stability results on the Lipschitz numerical index of vector-valued function space $A(\Omega, X)$. More precisely, we show in Section \ref{subsection:omega} that $n_L (X) \leq n_L (A(\Omega, X))$ under a certain assumption on the set of strong peak points of the base algebra. In Section \ref{subsec:A(K,X)}, we show that same inequality also holds for the space $A(K, X)$. Here we use the Choquet boundary of the base algebra instead of the set of strong peak points. Furthermore, we also discuss when the equality of these Lipschitz numerical radii holds for the spaces $A(\Omega, X)$ and $A(K, X)$. As an application of the results from Section \ref{subsection:omega}, assuming the set of strong peak points is a norming set, we show in Section \ref{subsection:DP} that $A(\Omega, X)$ inherits the Daugavet property from $X$. A similar result on $A(K, X)$ will be also provided in the same section. 

In Section \ref{section:absolute-sum}, we study the stability results of Lipschitz numerical index especially on the absolute sums of Banach spaces (Section \ref{subsection:absolute-sum}), the K\"othe-Bochner spaces (Section \ref{subsection:Kothe}), and the Banach spaces which contain a dense increasing family of one-complemented subspaces (Section \ref{subsection:increasing}). As a consequence, we show in Section \ref{subsection:ell_p-sum} that the Lipschitz numerical index of $\ell_p (X)$ is the limit of the sequence of the Lipschitz numerical index of $\ell_p^m (X)$, and the Lipschitz numerical index of Lebesgue-Bochner space $L_p(\mu, X)$ coincides with that of $\ell_p(X)$ for any Banach space $X$ by using the stability result on the $\ell_p$-sum. 

\section{Lipschitz numerical radii, indices, and Daugavet centers}

In this section, we provide the Lipschitz analogues of the properties that are well-known for classical numerical radii and indices and their connection to the numerical radius attaining operators and the Daugavet centers. 

\subsection{Lipschitz numerical index and renorming}\label{subsection:renorming}

Let $\mathcal{E}(X)$ be the set of all equivalent norms on $X$. If we equip this set with the following metric 
\[
\mu (p, q) = \log \left( \min \left\{ k \geq 1 : p \leq k q, \, q \leq k p \right\} \right), 
\]
then the set $(\mathcal{E} (X), \mu)$ is known to be arcwise connected \cite[Theorem 1]{RA}. We denote by $M(E)$ the space of all nonempty, bounded, closed subsets of a metric space $E$ endowed with the Hausdorff metric $d$ that is given by 
\[
d(A,B) = \max\left\{ \sup_{x \in A} d(x, B), \sup_{x \in B} d(x,A)\right \}\quad \text{for $A, B \in M(E)$}. 
\]

Let $(X, p)$ be a Banach space equipped with the norm $p$. 
For convenience, we denote also by $p$ the dual norm of $p$. 
We define the set $\Pi_p$ by
\[
\Pi_p = \{(x, y, x^*) \in X \times X \times X^*: x^* (x-y) = p(x^*) p (x-y) = \{p(x-y)\}^2 \}.
\]
and {\it Lipschitz numerical radius $\nu_L^p (T)$ of $T\in \Lip(X, X)$ with respect to the norm $p$} by 
\[
\nu_L^p (T) = \sup \left\{ \left| \frac{x^* (Tx - Ty)}{\{p(x-y)\}^2} \right| : (x, y, x^*) \in \Pi_p \right\}. 
\]
This allows us to define the \emph{Lipschitz numerical index} $n_L^p(X)$ \emph{with respect to the norm $p$} by
$$
n_L^p(X) = \inf \{\nu_L^p(T): T \in S_{\Lip((X,p),(X,p))}\}
$$
and the set $\mathcal{N}_L (X)$ of values of Lipschitz numerical index that $X$ can have up to renorming, i.e., 
\[
\mathcal{N}_L (X) := \{ n_L^p (X) : p \in \mathcal{E} (X)\}. 
\]

Our aim is to prove that $\mathcal{N}_L (X)$ is an interval for every Banach space $X$. To this end, we first prove that the map from an equivalent norm on $X$ to the Lipschitz numerical index with respect to the same norm is continuous with respect to the metric $\mu$. The similar results with respect to the numerical radius and index are shown in \cite[Corollary 18.4]{BD2} (see also \cite[Proposition 2]{FMP}). 

\begin{Proposition}\label{Prop:unifconti}
	Let $X$ be a Banach space and $T \in \Lip (X, X)$. Then the mapping $p \mapsto \nu_L^p (T)$ from $\mathcal{E} (X)$ to $\mathbb{R}$ is uniformly continuous on each bounded subset of $\mathcal{E} (X)$. 
\end{Proposition} 

\begin{proof}
	Let $G_\kappa := \{ p \in \mathcal{E} (X) : \mu (p, \| \cdot\| ) < \kappa \}$ and $\eps >0$. It is well-known from \cite[Theorem 3]{BD2} that there exists $\delta >0$ such that 
	\begin{equation}\label{eq:pqapprox}
		d(\Pi_p , \Pi_q) < \eps \quad \text{whenever } \mu (p,q) < \delta \text{ with } p,q \in G_\kappa.
	\end{equation}
		Now, fix $p, q \in G_k$ satisfying (\ref{eq:pqapprox}). For a given $\eta >0$, choose $(x,y,x^*) \in \Pi_q$ so that 
	\[
	\frac{ | x^* (Tx-Ty)|}{\{q(x-y)\}^2} > \nu_L^q (T) - \eta. 
	\]
	Since $d(\Pi_p , \Pi_q) < \eps$, we can also find $(u, u^*) \in \Pi_p$ such that 
	\[
	\left\| u -\frac{x-y}{q(x-y)} \right\| + \left\| u^* - \frac{x^*}{q(x^*)} \right\| < \eps. 
	\]
	If we let $x' = y + q(x-y)u$ and $y' = y$, it is easy to see from the direct computation that 
	$\| x' - x \| < q(x-y) \eps$.
	Furthermore, notice that $(x', y', q(x-y)u^*) \in \Pi_p$. 
	Then, we obtain 
	\begin{align*}
		&\left| \frac{x^* (Tx- Ty)}{\{q(x-y)\}^2} - \frac{[q(x-y)u^* ](Tx'-Ty')}{\{p(x'-y')\}^2} \right| \\
		&\leq \left| \frac{x^* (Tx') - x^* ( Tx)}{\{q(x-y)\}^2}\right| + \left| \frac{x^* (Tx') - [q(x-y)u^* ](Tx')}{\{q(x-y)\}^2} - \frac{x^* (Ty) - [q(x-y)u^* ](Ty)}{\{q(x-y)\}^2} \right| \\
		&\leq \frac{\|x^*\| \|T\|_L q(x-y)\eps }{\{q(x-y)\}^2} + \left| \frac{ [x^* - q(x^*) u^*](Tx'-Ty)}{\{ q(x^*) \}^2} \right| \quad (\text{since } q(x-y)=q(x^*)) \\
		&\leq e^\kappa \eps {\|T\|_L} + \frac{\{ q(x^*) \}^2 \|T\|_L \|u\| \eps}{\{ q(x^*) \}^2} \\
		&= 2 \eps e^{\kappa} \|T\|_L. 
	\end{align*} 
	This implies that 
	$\nu_L^q (T) \leq \nu_L^p (T) + 2 \eps e^{\kappa} \|T\|_L + \eta.$
	Since $\eta >0$ was arbitrarily chosen, we have
	\[
	\nu_L^q (T) \leq \nu_L^p (T) + 2 \eps e^{\kappa} \|T\|_L.
	\]
	By using the same argument, we can also show that $\nu_L^p (T) \leq \nu_L^q (T) + 2 \eps e^{\kappa} \|T\|_L$. Therefore, we have 
	\[
	|\nu_L^p (T) - \nu_L^q (T)| \leq 2 \eps e^{\kappa} \|T\|_L, 
	\]
	which proves our claim.
\end{proof}

\begin{Proposition}
	Let $X$ be a Banach space. Then the mapping $p \mapsto n_L^p (X)$ from $\mathcal{E}(X)$ to $\mathbb{R}$ is continuous. Hence the set $\mathcal{N}_L (X)$ is an interval. 
\end{Proposition} 

\begin{proof}
	We follow the argument in \cite[Proposition 2]{FMP}. From Proposition \ref{Prop:unifconti}, we deduce that the mapping $(p, T) \in \mathcal{E} (X) \times \Lip (X, X) \mapsto \nu_L^p (T) \in \mathbb{R}$ is uniformly continuous on bounded sets. Let $p_0 \in \mathcal{E} (X)$, $B$ be an open ball centered at $p_0$, and $\mathcal{S} = \{T \in \Lip (X,X) : \|T\|_{p_0} = 1\}$ where $\| \cdot\|_{p_0}$ is understood as the Lipschitz norm on the space $(X, p_0)$.
	Since $\inf \{ \|T\|_p : p \in B, T \in \mathcal{S} \} > 0$, the mapping $\Psi : B \times \mathcal{S} \rightarrow \mathbb{R}$ given by 
	\[
	\Psi (p, T) = \frac{\nu_L^p (T)}{\|T\|_p}
	\]
	is uniformly continuous; hence $p \in \mathcal{E} (X) \mapsto n_L^p (X) = \inf \{ \Psi (p, T) : T \in \mathcal{S}\} \in \mathbb{R}$ is continuous on $B$.
\end{proof}

\subsection{Lipschitz numerical radius attaining maps}\label{subsection:attaining}

Recall that for a Banach space $X$, a bounded linear operator $T \in \Lin (X, X)$ is said to be a \emph{numerical radius attaining operator} when the supremum in the equation (\ref{eq:classicnr}) is attained by an element $(x, x^*)$ such that $x^*(x) = 1$. Numerical radius attaining operators, introduced in \cite{Sims}, have been investigated by many researchers (see, for instance, \cite{AcoAguPay-Carolinae, C2, Paya1,CapMM} and the references therein). In a similar spirit, we now define a {Lipschitz numerical radius attaining} map. A Lipschitz map $T \in \Lip (X,X)$ \emph{attains its Lipschitz numerical radius} when the supremum in the equation (\ref{eq:Lipnr}) is attained by an element $(x, y, x^*) \in \Pi$.

First, we prove that the set of Lipschitz numerical radius attaining maps on $X$ is not dense in $\Lip (X, X)$ for every Banach space $X$. In order to prove this, we need the following lemma. 

\begin{Lemma}\label{Prop:line}
	Let $X$ be a Banach space. If $T \in \Lip (X, X)$ attains its Lipschitz numerical radius at $(x,y, x^*) \in \Pi$, then $T$ attains its Lipschitz numerical radius at $(x,z,(1-\lambda)x^*) \in \Pi $ and at $(z,y, \lambda x^*) \in \Pi$ where $z = \lambda x + (1-\lambda )y$ with $0< \lambda < 1$. 
\end{Lemma} 

\begin{proof}
	Since $(1-\alpha)x^* \in D(x-z)$ and $\alpha x^* \in D(z-y)$, we have 
	\begin{align*}
		\nu_L (T) = \frac{x^* (Tx- Tz) + x^* (Tz- Ty)}{ \|x-y\|^2} &= \frac{(1-\alpha)x^* (Tx-Tz)}{\|x-z\|^2} + \frac{\alpha x^* (Tz- Ty)}{\|z-y\|^2} \\
		&\leq (1-\alpha) \nu_L (T) + \alpha \nu_L (T) \\
		&= \nu_L (T). 
	\end{align*} 
	Hence, $T$ attains its Lipschitz numerical radius at $(x,z,(1-\lambda)x^*) \in \Pi $ and at $(z,y, \lambda x^*) \in \Pi$.
\end{proof}

\begin{Theorem}
	Let $X$ be a Banach space. Then the set of Lipschitz numerical radius attaining maps is not dense in $\Lip (X, X)$. 
\end{Theorem} 

\begin{proof}
	We follow the argument in \cite[Theorem 2.3]{KMS}. Fix $x_0 \in S_X$ and consider the line segment $[0,x_0] = \{ \lambda x_0 : 0 \leq \lambda \leq 1\}$. We may assume for convenience that $[0,x_0] =[0,1] \hookrightarrow X$. Let $P : X \rightarrow [0,1]$ be a norm one retraction Lipschitz map. Define $S \in \Lip ([0,1], \mathbb{R})$ by 
	\[
	S(t) = \int_0^t \chi_C (\tau) \, d\tau,
	\]
	where $C$ is a nowhere dense closed subset of $[0,1]$ with a positive measure. Then $\|S \|_L = 1$. Put $R:= E S P : X \rightarrow X$ where $E$ is the natural injection map from $[0,1]$ to $X$. We have then $R \in \Lip (X, X)$ and $\|R \|_L = \nu_L(R) =1$. Suppose there exists $T \in \Lip (X,X)$ which attains its Lipschitz numerical radius and satisfies that $\| T - R \|_L < 1/2$. It follows that $\nu_L (T- R) < 1/2$ and $\nu_L (T) \geq \nu_L (R) - \|T-R\|_L > 1/2$. Let us say, $T$ attains its Lipschitz numerical radius at $(x,y,x^*) \in \Pi$. By using Lemma \ref{Prop:line}, we have that 
	\begin{align*} 
		|x^* (Rz_1 - Rz_2)| &\geq |x^* (Tz_1-Tz_2)| - \nu_L (T- R) \|z_1-z_2\|^2 \\
		&=\nu_L (T) \|z_1 - z_2\|^2 - \nu_L (T-R) \|z_1-z_2\|^2 \\
		&> 0
	\end{align*} 
	for all $z_1 \neq z_2 \in [x,y]$. This implies that $R z_1 \neq Rz_2$ for all $z_1 \neq z_2 \in [x,y]$, which is impossible.
\end{proof}

As a generalization of classical Lebesgue's differentiability theorem, it is well-known that if $X$ is a separable Banach space and $Y$ is a Banach space with the RNP, then every Lipschitz map $T$ from $X$ into $Y$ is G\^ateaux differentiable outside an Ar\'onszajn null set (see \cite[Theorem 6.42]{BL}). Our next result, which is motivated by \cite[Proposition 3.2]{WHT}, shows that the Lipschitz numerical radius attainment of $T \in \Lip (X,X)$ is related to the classical numerical radius attainment of its G\^ateaux derivative.
We write $D_T (u)$ for the G\^ateaux derivative of $T$ at $u$, i.e., 
\[
D_T (u) (x) = \lim_{t \rightarrow 0} \frac{T (u + tx) - T(u)}{t} \quad (x \in X). 
\]

\begin{Proposition}
	Let $X$ be a Banach space. If $T \in \Lip (X,X)$ attains its Lipschitz numerical radius at $(x,y, x^*) \in \Pi$ and $T$ is G\^ateaux differentiable on $[x,y]$, then $D_T (u) \in \mathcal{L} (X, X)$ attains its numerical radius at $\left( \frac{x-y}{\|x-y\|}, \frac{x^*}{\|x^*\|}\right)$ for every $u \in [x, y]$.
\end{Proposition} 

\begin{proof}
	Let $u \in [x,y]$ be fixed. Note that 
	\[
	D_T (u) (z_0) = \lim_{t \rightarrow 0} \frac{T(u + tz_0) - T(u)}{t},
	\]
	where $z_0 = (x-y)/\|x-y\|$. Given $\eps >0$, let $\delta >0$ be such that 
	\[
	0<|t| < \delta \Longrightarrow \left\| D_T (u) (z_0) - \frac{T(u+tz_0) - T(u)} {t} \right\| < \eps. 
	\]
	Let $\{ x_1,\ldots, x_{N-1} \} \subset [x,y]$ be a set of finite points such that $x_j = t_jx+(1-t_j)y$ where $1=t_N > t_{N-1} > \ldots > t_1>t_0=0$ with $\|x_j - x_{j+1} \| < \frac{\delta}{2}$ for every $j =0,\ldots, N-1$ and $x_{i} = u$ for some $i \in \{0,\ldots, N\}$. Then 
	\begin{equation}\label{eq:D_T(u)}
		\left\| D_T (u) (z_0) - \frac{T x_{i+1} - Tu }{\| x_{i+1} - u\|} \right\| < \eps. 
	\end{equation} 
	On the other hand, note from Lemma \ref{Prop:line} that 
	\[
	\left| \frac{x^*}{\|x^*\|} \left( \frac{Tx_j - Tx_{j+1}}{\|x_j - x_{j+1}\|} \right) \right| = \nu_L (T)
	\]
	for each $j=0,\ldots, N-1$. In particular, from \eqref{eq:D_T(u)}, we have that 
	\[
	\left| \frac{x^*}{\|x^*\|} \left( D_T (u) (z_0) \right) \right| \geq \left| \frac{x^*}{\|x^*\|} \left( \frac{T x_{i+1} - Tu }{\| x_{i+1} - u\|} \right) \right| -\eps = \nu_L (T) -\eps. 
	\]
	By letting $\eps \rightarrow 0$ and by \cite[Proposition 3.2]{WHT}, 
	\[
	\nu_L (T) \geq \nu (D_T (u) ) \geq \left| \frac{x^*}{\|x^*\|} \left( D_T (u) (z_0) \right) \right| \geq \nu_L (T)
	\]
	which implies that $D_T (u)$ attains its numerical radius at $(z_0, \frac{x^*}{\|x^*\|})$. 
\end{proof}

\subsection{Lipschitz Daugavet centers}\label{subsection:center}

As a Lipschitz analogue of the Daugavet centers, we now introduce the Lipschitz Daugavet centers.

\begin{Definition}
	Let $X$ and $Y$ be Banach spaces. A Lipschitz map $G \in \Lip (X, Y)$ is said to be a \emph{Lipschitz Daugavet center} if $\|G + T \|_L = \|G \|_L + \|T \|_L$ for every rank-one Lipschitz map $T \in \Lip (X, Y)$. 
\end{Definition} 

Our aim is to extend the result \cite[Corollary 3.6]{KMMW} by showing that if $G \in \mathcal{L} (X, Y)$ is a Daugavet center, then $\| G + T \|_L = \|G\| + \|T \|_L$ for every $T \in \Lip (X, Y)$ such that $\overline{\conv} (\slope (T))$ has one of the properties: Radon-Nikod\'ym property, Asplund property, CPCP, or absence of $\ell_1$-sequences. Here the {\it slope} of a Lipschitz map $T$ is defined by 
$$
\slope(T) := \left\{ \dfrac{Tu-Tv}{\|u-v\|} : u,v \in X,\, u \neq v \right\}.
$$
As an immediate corollary, we will obtain that $G \in \Lin (X, Y)$ is a Daugavet center if and only if $G$ is a Lipschitz Daugavet center. To this end, we collect first some generalizations to Daugavet centers of the properties for the identity $\Id_X$ on a Banach space $X$ with the Daugavet property. 

Recall from \cite{KMMW} that for a Banach space $X$, a \emph{Lip-slice} of $S_X$ is a non-empty set of the form 
\[
\left\{ \frac{u -v}{\|u-v\|} : u,v \in X, \, u \neq v, \, \frac{f(u)- f(v)}{\|u-v\|} > \alpha \right\} 
\]
for $f \in \Lip (X, \mathbb{R})$ and $\alpha \in \mathbb{R}$. We shall denote the set 
\[
S_L(f,\eps) = \left\{ \frac{u -v}{\|u-v\|} : u,v \in X,\, u\neq v, \, \frac{f(u)- f(v)}{\|u-v\|} > \|f\|_L - \eps \right\} 
\]
for $f \in \Lip (X, \mathbb{R})$ and $\eps >0$. 

The following result generalizes \cite[Corollary 2.6]{KMMW}. 

\begin{Proposition}\label{prop:center1}
	Let $X$ and $Y$ be Banach spaces, and let $G \in \mathcal{L} (X, Y)$. Then $G$ is a Daugavet center if and only if for any $y \in S_Y$, $f \in \Lip (X, \mathbb{R})$ and $\eps >0$, there is $u \in S(f, \eps)$ such that $\|G u + y \| > 2 -\eps$. 
\end{Proposition} 

\begin{proof}
	It suffices to prove `only if' direction. Suppose that $\|G \| = 1$. Let $y_0 \in S_Y$ and $\eps >0$ be fixed. Note from \cite[Theorem 2.5]{BK} that
	\[
	B_X \subseteq \overline{\conv} (B_X \setminus G^{-1} (B_Y (-y_0, 2-\eps)) ),
	\]
	where $B_Y (-y_0, 2-\eps)$ denotes the closed ball with radius $2-\eps$ centered at $-y_0$ in $Y$. 
	It follows that 
	\[
	S(f, \eps) \cap \overline{\conv} (B_X \setminus G^{-1} (B_Y (-y_0, 2-\eps)) ) \neq \O;
	\]
	hence by \cite[Lemma 2.4]{KMMW}, we get that 
	\[
	S(f, \eps) \cap (B_X \setminus G^{-1} (B_Y (-y_0, 2-\eps)) ) \neq \O.
	\]
	Take $z:=(u-v)/\|u-v\|$ to be an element in this intersection, then $\|G z + y_0\| > 2- \eps$. 
\end{proof} 

Before going further, we clarify some notation. Given a Banach space $X$, let us denote by $K(X^*)$ the intersection of $S_{X^*}$ with the weak-star closure in $X^*$ of the set $\textup{Ext} (B_{X^*})$ of extreme points of $B_{X^*}$. That is, $K(X^*) = S_{X^*} \cap \overline{\textup{Ext} (B_{X^*})}^{w^*}$. We will consider $K(X^*)$ as a topological space equipped with the weak-star topology. Then $K(X^*)$ is a Baire space \cite[Lemma 3.2]{KMMW}. Given Banach spaces $X$ and $Y$, a Lip-slice $S = S(f,\eps)$ of $S_X$, $G \in \mathcal{L} (X, Y)$ with $\|G\|=1$ and $\eps >0$, we consider the set 
\[
D_G (S, \eps) = \{ y^* \in K(Y^*) : \re y^* (Gx) > 1 - \eps \text{ for some } x \in S \}. 
\] 

\begin{Lemma}\label{Lem:D_G}
	Let $X$ and $Y$ be Banach spaces, $S(f, \eps)$ a Lip-slice of $S_X$, $G \in \mathcal{L} (X, Y)$ with $\|G \| = 1$ and $\eps >0$. Then $D_G (S, \eps)$ is an open subset of $K(Y^*)$. If, in addition, $G$ is a Daugavet center, then $D_G (S, \eps)$ is dense in $K(Y^*)$. 
\end{Lemma} 

\begin{proof}
	First, note that $D_G (S, \eps) = \bigcup_{x \in S} D_x$, where $D_x = \{ y^* \in K(Y^*) : \re y^* (Gx) > 1 -\eps\}$; hence $D_G(S,\eps)$ is weak-star open as it is a union of weak-star open sets of $K(Y^*)$. 
	Next, assume that $G$ is a Daugavet center. With the same proof of \cite[Lemma 3.2]{KMMW} but using Proposition \ref{prop:center1} instead of \cite[Corollary 2.6]{KMMW}, we can prove that $D_G (S, \eps)$ is dense in $K(Y^*)$.
\end{proof} 

Recall from \cite{AKMMS} that a bounded subset $C$ of a Banach space $X$ is a \emph{slicely countably determined set} (for short, SCD-set) if there is a \emph{determining sequence} $\{S_n\}$ of slices of $C$, that is, $\{S_n\}$ satisfies that $C \subseteq \overline{\conv} (B)$ whenever $B \subseteq C$ intersects all the $S_n$'s. 
The following result generalizes \cite[Theorem 3.4]{KMMW} while the proof is based on the same idea.

\begin{Theorem}\label{Thm:SCD}
	Let $X$ and $Y$ be Banach spaces and $G \in \mathcal{L} (X, Y)$. If $G$ is a Daugavet center, then $\|G + T \|_L = \|G \| + \|T \|_L$ for every $T \in \Lip (X, Y)$ such that $\slope (T)$ is an SCD-set. 
\end{Theorem} 

\begin{proof} 
Without loss of generality, let us assume that $\|G \| =1$ and consider the case of $\|T \|_L = 1$. For $\eps >0$, let $u \neq v \in X$ such that 
	\[
	{\|T u - Tv \|}> ( 1 - \eps){\|u-v\|} 
	\]
	Take a sequence $\{S_n\}$ of slices of $\slope (T)$ which is determining and write 
	\[
	S_n = \{ z \in \slope (T) : \re x_n^* (z) > 1 -\eps_n \}
	\]
	with $\eps_n >0$ and $x_n^* \in X^*$ such that $\sup \re x_n^* (\slope (T)) = 1$. For each $n \in \N$, let us denote by $\tilde{S}_n$ the set 
	\[
	\left \{ \frac{u-v}{\|u-v\|} : u \neq v, \, \frac{ \re x_n^* (Tu) - \re x_n^* (Tv) }{\|u-v\|} > 1 - \eps_n \right\}. 
	\] 
	By Lemma \ref{Lem:D_G} and the Baire category theorem, we obtain that the set 
	$A = \bigcap_{n \in \N} D_G (\tilde{S}_n , \eps)$ is dense in $K(Y^*)$. Therefore, there exists $y^* \in A$ such that 
	\[
	\re y^* \left( \frac{Tu - Tv}{\|u-v\|} \right) > 1-\eps. 
	\]
	Since $y^* \in A$, for each $n \in \N$, there exist $u_n \neq v_n$ in $X$ such that 
	\[
	\re y^* \left( \frac{G u_n - G v_n}{\|u_n - v_n\|} \right) > 1- \eps \quad \text{and} \quad \frac{ Tu_n - Tv_n }{\|u_n -v_n \|} \in S_n.
	\]
	Since $\{S_n\}$ is determining for $\slope (T)$, we then have 
	\[
	\slope (T) \subseteq \overline{\conv} \left\{ \frac{ Tu_n - Tv_n }{\|u_n -v_n \|} : n \in \N \right\}. 
	\]
	By a standard convexity argument as in the proof of \cite[Theorem 3.4]{KMMW}, there exists $\ell \in \N$ such that 
	\[
	\re y^* \left( \frac{Tu_\ell - Tv_\ell}{\|u_\ell - v_\ell\|} \right) > 1 - \eps.
	\]
	Finally, 
	\begin{align*}
		\|G + T \|_L &\geq \re y^* \left( \frac{G u_\ell - G v_\ell}{\| u_\ell - v_\ell \|} + \frac{ Tu_\ell - T v_\ell}{\|u_\ell - v_\ell\|} \right) > 2-2\eps. 
	\end{align*} 
\end{proof} 

By definition, if $\slope (T)$ is an SCD-set for $T \in \Lip (X, Y)$, then $T$ is forced to have a separable image. Nevertheless, the following result will be a tool to get rid of this separability restriction in many cases. The proof is similar to the proof of \cite[Lemma 3.5]{KMMW}, but slightly different. We include the proof for the sake of completeness. 

\begin{Proposition}\label{Prop:sep_det}
	Let $X$ and $Y$ be Banach spaces, $G \in \mathcal{L} (X, Y)$ and $T \in \Lip (X, Y)$. Then there exist separable subspaces $E$ and $F$ of $X$ and $Y$, respectively, such that $T \vert_E \in \Lip (E, F)$ and $G \vert_E \in \Lin (E, F)$ satisfy $\|T \vert_E\|_L = \|T\|_L$, $\| G\vert_E \| = \|G\|$ and $G \vert_E$ is a Daugavet center. 
\end{Proposition} 

\begin{proof}
	Let $\{x_k\}$ and $\{x_k'\}$ be sequences in $X$ such that $\| T \vert_{\overline{\operatorname{span}}\{x_k : k\in\N\} } \|_L = \|T \|_L$ and $\| G \vert_{\overline{\operatorname{span}}\{x_k':k\in \N\}} \| = \|G \|$, respectively. Let $X_1 := \overline{\operatorname{span}} \{ x_k, x_k' : k \in \N \}$ and $Y_1 := \overline{T(X_1)}$. Then $X_1$ and $Y_1$ are separable subspaces of $X$ and $Y$, respectively. By \cite[Theorem 1]{I}, there exist separable subspaces $E_1$ and $F_1$ of $X$ and $Y$, respectively, such that $X_1 \subseteq E_1$, $Y_1 \subseteq F_1$, $G(E_1) \subseteq F_1$ and $G \vert_{E_1} : E_1 \rightarrow F_1$ is a Daugavet center. Let $Y_2 := \overline{\operatorname{span}} \{ F_1 \cup T(E_1) \}$. Again, by \cite[Theorem 1]{I}, there exist separable subspaces $E_2$ and $F_2$ of $X$ and $Y$, respectively, such that $E_1 \subseteq E_2$, $Y_2 \subseteq F_2$, $G(E_2) \subseteq F_2$ and $G \vert_{E_2} : E_2 \rightarrow F_2$ is a Daugavet center. Let $Y_3:= \overline{\operatorname{span}} \{ F_2 \cup T(E_2)\}$. In this fashion, we can define a sequence $\{E_n\}$ and $\{F_n\}$ of separable subspaces of $X$ and $Y$, respectively, so that 
	\[
	X_1 \subseteq E_1 \subseteq E_2 \subseteq \cdots \subseteq E_n \subseteq \cdots 
	\]
	and
	\[
	Y_1 \subseteq F_1 \subseteq Y_2 \subseteq F_2 \subseteq \cdots \subseteq Y_n \subseteq F_n \subseteq \cdots. 
	\]
	Put $E := \cup_n E_n$ and $F := \cup_n F_n$. Then $T \vert_E \in \Lip (E, F)$ satisfies that $\|T \vert_E \|_L = \|T \|_L$. Moreover, $G(E) \subseteq F$ and $\| G \vert_E \| = \| G \|$. By using \cite[Theorem 2.1]{BK}, it can be checked that $G \vert_E : E \rightarrow F$ is a Daugavet center. 
\end{proof}

Now, we are ready to present the promised result which generalizes \cite[Corollary 3.6]{KMMW}. 

\begin{corollary}
	Let $X$ and $Y$ be Banach spaces and $G \in \mathcal{L} (X, Y)$. If $G$ is a Daugavet center, then $\|G + T \|_L = \|G \| + \|T \|_L$ for every $T \in \Lip (X, Y)$ such that $\overline{\conv} (\slope (T))$ has one of the following properties: Radon-Nikod\'ym property, Asplund property, CPCP, or absence of $\ell_1$-sequences. 
\end{corollary} 

\begin{proof}
	Let us take separable subspaces $E$ and $F$ from Proposition \ref{Prop:sep_det}. Observe that $\overline{\conv} (\slope (T \vert_E))$ also enjoys one of the properties listed above; hence $\overline{\conv} (\slope (T \vert_E))$ is an SCD-set since it is separable. By \cite[Lemma 3.1]{KMMW}, $\slope (T \vert_E)$ turns to be an SCD-set. It follows from Theorem \ref{Thm:SCD} that
	\[
	\| G + T \|_L \geq \|G \vert_E + T \vert_E \|_L = \|G \vert_E \| + \|T \vert_E \|_L = \|G\| + \|T\|_L. 
	\]
\end{proof}

We finish this section by mentioning that the lack of linearity induces a difficulty in estimating the lower bound of the Lipschitz numerical index of a Banach space with so called the property $\beta$. For the case of classical numerical index, see \cite[Proposition 5]{FMP}.

\begin{question}
Suppose that  $X$ is a Banach space with the property $\beta$. Is it possible to estimate the Lipschitz numerical index of $X$ in terms of the constant involved in the definition of the property $\beta$? 
\end{question}

\section{Lipschitz numerical index of vector-valued function spaces}\label{section:function-algebra}

In this section, we present the stability results of the Lipschitz numerical index of vector-valued function spaces. Recall that a {\it function algebra} $A(\Omega)$ is a closed subspace of $C_b(\Omega)$ that separates points of $\Omega$ and contains constant functions. When $\Omega$ is a compact Hausdorff space, the function algebra $A(\Omega)$ is known as the uniform algebra. We will use the notation $K$ instead of $\Omega$ in this case.

Let us emphasize that all Banach spaces are assumed to be complex in the present section. As a vector-valued analogue of a function algebra, we denote by $A(\Omega, X)$ a closed subspace of the space $C_b(\Omega, X)$ of $X$-valued bounded continuous functions that satisfies the following properties:

\begin{enumerate}[\rm(i)]
	\setlength\itemsep{0.3em}
	\item[\textup{(i)}] The base algebra $A = \{x^* \circ f: f \in A(\Omega, X),\, x^* \in X^* \}$ is a function algebra.
	\item[\textup{(ii)}] $A\otimes X \subset A(\Omega, X)$.
	\item[\textup{(iii)}] For every $f \in A$ and $g \in A(\Omega, X)$, $fg \in A(\Omega, X)$.
\end{enumerate}

In order to compare the Lipschitz numerical index of $A(\Omega, X)$ to that of $X$, we apply the following lemmata from \cite{WHT} which were used to show the stability of the Lipschitz numerical index of $C(K, X)$. 

\begin{Lemma}\cite[Corollary 2.2]{WHT}\label{lem:radiusattain}
	Let $Q \subset \Pi$ be such that $\pi(Q)$ is dense in $X \times X$. Then for each $T \in \Lip(X, X)$, 
	\[
	\nu_L(T) = \sup\left\{\frac{|x^*(Tx - Ty)|}{\|x-y\|^2}: (x,y,x^*) \in Q\right\}.
	\]
	Moreover, $\nu_L(T)$ may be determined by choosing one functional $x^* \in D(x-y)$ at each point $(x,y)$ of a dense subset of $X \times X$.
\end{Lemma}

The {\it join} $J(A, B)$ of sets $A$ and $B$ is the union of all line segments from these sets and defined by $J(A,B) = \bigcup\{ \lambda a + (1-\lambda) b : a \in A, b \in B, 0\leq \lambda \leq 1\}$. 
 
\begin{Lemma}\cite[Lemma 2.6]{WHT}\label{lem:specialh}
	Let $X$ be a Banach space and $x, y \in X$ with $x \neq y$. If $\frac{x-y}{\|x-y\|} \in \overline{J(E)}$ for some set $E \subset B_X$, then for every $\eps > 0$, there exists $z \in X$ such that
	\[
	z - x \in \frac{\|x-y\|}{2}E\,\,\, \text{and} \,\,\, \|y - z\|\leq (1 + \eps) \frac{\|x -y\|}{2}.
	\] 
\end{Lemma}

\begin{Lemma}\cite[Lemma 4.1]{WHT}\label{lem:Mlip}
	Let $X$ and $Y$ be Banach spaces. If $x_1, x_2 \in X$, $y_1, y_2 \in Y$ satisfy that $\|y_1 - y_2\|_Y \leq M \|x_1 - x_2\|_X$ for some $M > 0$, then there exists a $M$-Lipschitz map $F \in \textup{Lip}(X, Y)$ such that $F(x_1) = y_1$ and $F(x_2) = y_2$.
\end{Lemma}

\subsection{Lipschitz numerical index of $A(\Omega,X)$}\label{subsection:omega}
%In this section, we show the stability of Lipschitz numerical index of the space $A(\Omega, X)$. 
 When we inspect the proof of the stability result on the Lipschitz numerical index of $C(K, X)$ in \cite{WHT}, Urysohn's lemma also plays an important role along with the aforementioned lemmata. However, it is shown that the Urysohn's lemma may be too rigid to use for function algebras (for such an example, see \cite[pp. 371]{CGK}). 
 
 To overcome this difficulty, a Urysohn-type lemma for function algebras is introduced in \cite{KL}, inspired by the similar lemma for uniform algebras (see \cite{CGK}). A point $t_0 \in \Omega$ is said to be a {\it strong peak point} for $A(\Omega)$ if there exists $f \in A(\Omega)$ such that $\|f\|_{\infty} = |f(t_0)|$ and for every open neighborhood $U$ containing $t_0$, we have $|f(t)| < \|f\|_{\infty}$ for all $t \notin U$. The collection of all strong peak points for $A(\Omega)$ is denoted by $\rho A(\Omega)$. Now we recall a Urysohn-type lemma for function algebras.
 
\begin{Lemma}\cite[Lemma 3]{KL}\label{lem:urysohnsp}
	Let $\Omega$ be a Hausdorff space and $A$ be a subalgebra of $C_b(\Omega)$. Suppose that $U$ is an open subset of $\Omega$ and $t_0 \in U \cap \rho A$. Then given $0< \eps < 1$, there exists a strong peaking function $\varphi \in A$ such that $\|\varphi\| = 1 = \varphi(t_0)$, $\sup_{t \in \Omega \setminus U} |\varphi(t)| < \eps$, and for all $ t \in \Omega$,
	\[
	|\varphi(t)| + (1- \eps)|1 - \varphi(t)| \leq 1.
	\]
\end{Lemma}

%The space $A_b(B_Y, X)$ (resp. $A_u(B_Y, X)$) of $X$-valued bounded (resp. uniformly) continuous functions on $B_Y$ that are holomorphic on the interior of $B_Y$ is an example of the space $A(\Omega, X)$. 
% We can also define a strong peak point for the space $A(\Omega, X)$ similar to a strong peak point for $A(\Omega)$ with appropriate modification. It is well-known that for the base algebra $A$ of the space $A(\Omega, X)$ we always have $\rho A = \rho A(\Omega, X)$ \cite{L}.

First, we compare the Lipschitz numerical index of $A(\Omega, X)$ with the index of $X$.	
	\begin{Theorem}\label{th:lipnumaox}
		Let $\Omega$ be a Hausdorff space, $X$ be a Banach space, and $A$ be the base algebra of $A(\Omega, X)$. Suppose that the set $\rho A$ of strong peak points is a norming subset of $\Omega$. Then $ n_L(X)\leq n_L(A(\Omega, X)) $.
	\end{Theorem}
	
	\begin{proof}
	Let $T \in \Lip(A(\Omega, X),A(\Omega, X))$ where $\|T\|_L = 1$. For every $t \in \Omega$, define $T_t \in \Lip(A(\Omega, X), X)$ by $T_t(h) = (Th)(t)$ for every $h \in A(\Omega,X)$. Notice that $\|T\|_L = \sup\{\|T_t\|_L: t \in \rho A\}$. Let $\eps >0$ be fixed. Find $f, g \in A(\Omega, X)$ and, since $\rho A$ is norming, $t_0 \in \rho A$ such that
		\begin{equation}\label{eq:fng}
			\|T_{t_0}(f) - T_{t_0}(g)\| > (1-\eps)\|f - g\|. 	
		\end{equation}
		 
		Put $f_0 = \frac{f-g}{\|f-g\|}$ and consider 
		\[
		C(t_0) = \{\varphi \in A : |\varphi(t)| + (1 - \eps)|1 - \varphi(t)| \leq 1 \text{ for every } t \in \Omega \text{ and } \varphi(t_0) = 1 \}
		\] 
		and 
		\begin{equation}\label{eq:Ef0}
		E_{f_0} = \{(1-\eps)(1 - \varphi)f_0 + \varphi x: x \in S_X, \,\varphi \in C(t_0)\}.
		\end{equation} 
		First, we show that $f_0 \in \overline{J(E_{f_0})}$. Indeed, for $\delta > 0$, let $f_0(t_0) = \lambda x_1 + (1 - \lambda) x_2$ for $x_1, x_2 \in S_X$ for some $0\leq \lambda \leq 1$. Define a set $U = \{t \in \Omega: \|f_0(t) - f_0(t_0)\| < \delta\}$. Then by Lemma \ref{lem:urysohnsp}, there exists $\varphi \in A$ such that $\varphi(t_0) = 1$, $\sup_{\Omega \setminus U} |\varphi(t)| < \delta$ and 
		\begin{equation}
		|\varphi(t)| + (1 - \delta)|1 - \varphi(t)| \leq 1.
		\end{equation}
		
		Now define 
		\begin{align*}
			f_j &= (1 - \delta)(1 - \varphi)f_0 + \varphi x_j \qquad \text{for } j=1,2.
		\end{align*}
		For $t \in \Omega$, we can see that 
		\begin{equation*}
			\|f_0(t) - \lambda f_1(t) - (1- \lambda) f_2(t)\|_X 
			\leq 2\delta + \|\varphi(t) (f_0(t) - f_0 (t_0)) \|_X.
		\end{equation*}
		Here we consider two cases: If $t \in U$ we have 
		\begin{align*}
		 \|\varphi(t) (f_0(t) - f_0 (t_0)) \|_X \leq \|f_0(t) - f_0(t_0)\|_X < \delta.
		\end{align*} 
		On the other hand, if $t \in \Omega \setminus U$ we have
		\[
		 \|\varphi(t) (f_0(t) - f_0 (t_0)) \|_X < \eps \|f_0(t) - f_0 (t_0) \|_X \leq 2\delta.
		\]
		This shows that
		\[
		\|f_0 -( \lambda f_1 + (1- \lambda) f_2) \|_X \leq 4\delta,
		\]
		hence, $f_0 \in \overline{J(E_{f_0})}$. Also notice that if $x \in S_X$ and $\varphi \in C(t_0)$, we have 
		\[
		\|(1-\eps)(1 - \varphi(t))f_0(t) + \varphi(t) x\|_X \leq (1 - \eps)|1 - \varphi(t)| + |\varphi(t)| \leq 1
		\]
		for every $t \in \Omega$. Hence we see that $E_{f_0} \subset B_{A(\Omega, X)}$.		
		Then by Lemma \ref{lem:specialh}, there exists $h \in A(\Omega, X)$ such that
		\begin{equation}\label{eq:specialh}
			h - f \in \frac{\|f - g\|}{2}E_{f_0}\,\,\, \text{and} \,\,\, \|h - g\| \leq (1 + \eps)\frac{\|f - g\|}{2}.
		\end{equation}
		Combining \eqref{eq:fng} with \eqref{eq:specialh}, we obtain that 
		\begin{equation*}
			\|T_{t_0}h - T_{t_0}f\|_X \geq \|T_{t_0}f - T_{t_0}g\| - \|g-h\| \geq (1 - 3\eps)\|h - f\|.
		\end{equation*}
		
		Now let $x_0 = h(t_0)$ and $y_0 = f(t_0)$. Observe that $\|x_0 - y_0\|_X = \|h - f\|$. By Lemma \ref{lem:Mlip}, there exists a norm one $F \in \textup{Lip}(X, A(\Omega, X))$ such that $F(x_0) = h$ and $F(y_0) = f$.
		Let 
		\begin{equation}\label{eq:setU}
			\tilde{U} = \{t \in \Omega : \|h(t) - x_0\| < \eps\|h-f\|\,\,\, \text{and} \,\,\, \|f(t) - y_0\| < \eps\|h-f\|\}.
		\end{equation}	
		Then by Lemma \ref{lem:urysohnsp}, there exists $\varphi \in A$ such that
		\small
		\begin{equation}\label{eq:urysohn2}
			\varphi(t_0) = 1, \,\,\, \sup_{\Omega \setminus \tilde{U}} |\varphi(t)| < \eta := \frac{\eps \|h-f\|}{2\|h\|+2\|f\|+1} \,\,\, \text{and}\,\,\, |\varphi(t)| + (1 - \eta)|1 - \varphi(t)| \leq 1.
		\end{equation}
		\normalsize
		Define a function $\Phi: X \rightarrow A(\Omega, X)$ by 
		\[
		\Phi(x) = (1-\eta)(1 - \varphi)F(x) + \varphi x.
		\]
		By the triangle inequality and \eqref{eq:urysohn2}, we obtain
		\begin{align*}
		\|T_{t_0}(\Phi(x_0)) &- T_{t_0}(\Phi(y_0))\|_X \\
		&\geq \|T_{t_0}(h) - T_{t_0}(f)\|_X - \|\eta\varphi h - \varphi h - \eta h + \varphi x_0\|	- \|\eta\varphi f - \varphi f - \eta f + \varphi y_0\| \\ 
		&\geq \|T_{t_0}(h) - T_{t_0}(f)\|_X - ( \eps \|h - f\| + \|\varphi h - \varphi x_0\| ) - ( \eps \|h - f\| + \|\varphi f - \varphi y_0\|) \\ 
		&= \|T_{t_0}(h) - T_{t_0}(f)\|_X - 2\eps\|h-f\| - \|\varphi h - \varphi x_0\| - \|\varphi f - \varphi y_0\|.
		\end{align*} 
		Again, we consider two cases. If $t \in \tilde{U}$, we have
		\[ 
		\|\varphi(t) h(t) - \varphi(t) x_0\|_X \leq \|h(t) - x_0\|_X \overset{(\ref{eq:setU})}{<} \eps\|h-f\|.
		\]
		On the other hand if $t \in \Omega \setminus \tilde{U}$, we have
		\[
		\|\varphi(t) h(t) - \varphi(t) x_0\|_X \overset{(\ref{eq:urysohn2})}{\leq} 2 \eta \|h\| < \eps\|h-f\|. 
		\]
		Hence we obtain that $\|\varphi h - \varphi x_0\| < \eps\|h - f\|$. Similarly, we can also show that $\|\varphi f - \varphi y_0\| < \eps\|h - f\|$. This implies that
		\begin{eqnarray*}
		\|T_{t_0}(\Phi(x_0)) - T_{t_0}(\Phi(y_0))\|_X 
		&\geq& \|T_{t_0}(h) - T_{t_0}(f)\|_X - 4\eps\|h-f\|\\
		&\geq& (1-3\eps)\|h - f\| - 4\eps\|h - f\| = (1 - 7\eps) \|x_0 - y_0\|.
		\end{eqnarray*}
	
		Now define $S \in \Lip(X, X)$ by $S(x) = T_{t_0}(\Phi(x)) - T_{t_0}(\Phi(0))$ for every $x \in X$. Since $\|S(x_0) - S(y_0)\|_X \geq (1 - 7\eps)\|x_0 - y_0\|$, we have that $1 - 7\eps \leq \|S\|_L \leq 1$.
				If $(x - y)^* \in D(x-y)$, then $(x-y)^* \circ \delta_{t_0} \in D(\Phi(x) - \Phi(y))$. Furthermore, 
		\[
		(x-y)^*\circ \delta_{t_0}(T(\Phi(x)) - T(\Phi(y))) = (x-y)^*(Sx - Sy).
		\] 
		This implies that $\nu_L (T) \geq \nu_L (S)$. Moreover, $n_L(X) \leq \nu_L\left(\frac{S}{\|S\|_L}\right) \leq \frac{1}{1-7\eps}\nu_L(S)$. Therefore, we obtain $\nu_L (T) \geq \nu_L (S) \geq (1-7\eps) n_L(X)$. Since $\eps > 0$ is arbitrary, we get $n_L(A(\Omega, X)) \geq n_L(X)$.
		\end{proof}

		Under a certain assumption, the Lipschitz numerical index of $A(\Omega, X)$ is equal to the Lipschitz numerical index of $X$.
		
		\begin{Theorem}\label{th:lipnumaox_conv}
		Let $\Omega$ be a Hausdorff space, $X$ be a Banach space, and $A$ be the base algebra of $A(\Omega, X)$. Suppose that the set $\rho A$ of strong peak points is a norming subset of $\Omega$. If the range of the mapping $Q_S : A(\Omega, X) \rightarrow C_b (\Omega, X)$ given by 
		\[ 
		Q_S (f) (t) = S ( f(t)) \quad (f \in A(\Omega,X),\, t \in \Omega)
		\]
		is contained in $A(\Omega, X)$ for every $S \in \Lip (X, X)$, then $ n_L(X) = n_L(A(\Omega, X)) $.
	\end{Theorem}
		
		\begin{proof} 
		It suffices to show that $n_L(X) \geq n_L(A(\Omega, X))$. Let us fix $S \in \Lip(X, X)$ with $\|S\|_L = 1$. By assumption, $T:= Q_S$ is a member of $\Lip (A(\Omega, X), A(\Omega,X))$ and $\|T \|_L = 1$. 
		For every $t \in \rho A$, define the following sets:
		\small
		\begin{align*}
			&\mathcal{A}_{t} = \{ (f,g): \|f(t) - g(t)\|_X = \|f-g\| \} \subset A(\Omega, X) \times A(\Omega, X) ,\\
			&Q = \{(f,g, f^* \circ \delta_t) : (f,g) \in \mathcal{A}_{t},\, f^* \in D(f(t) - g(t)),\, t \in \rho A\} \subset A(\Omega, X) \times A(\Omega, X) \times A(\Omega, X)^*.
		\end{align*}
		\normalsize
		Note that $Q \subset \Pi_{A(\Omega, X)}$. We claim that $\pi(Q)$ is dense in $A(\Omega, X) \times A(\Omega, X)$. As a matter of fact, let $(f,g) \in A(\Omega, X) \times A(\Omega, X)$ and $\delta >0$ be given. Choose $t_0 \in \rho A$ such that 
		\[
		\|f-g\|^{-1} \| (f-g)(t_0) \| > 1 - \frac{\delta}{5\|f-g\|}.
		\] By \cite[Theorem 4]{L}, there exists $h \in A(\Omega, X)$ with $\| h \| = 1$ which strongly peaks at $t_0$ and satisfies that $\| h - \|f-g\|^{-1} (f-g) \| < \delta \|f-g\|^{-1}$. Write $\| f-g\| h = f' - g'$ where $f' = \|f-g\| h + g$ and $g'=g$. Then $(f' , g') \in \mathcal{A}_{t_0}$ and $\max \{\| f' - f \| , \|g'-g\| \} < \delta$, which completes the claim. 
		By Lemma \ref{lem:radiusattain} we have
		\[
		\nu_L(T) = \sup\left\{\frac{|(f^* \circ \delta_t)(Tf - Tg)|}{\|f-g\|^2}: (f, g, f^* \circ \delta_t) \in Q\right\}.
		\]
		Then from the definitions of $\mathcal{A}_t$ and $Q$, for given $\eps > 0$, there exists $(f,g, f^*\circ\delta_t) \in Q$ such that 
		\begin{eqnarray*}
		n_L(T) - \eps \leq \nu_L(T) - \eps < \frac{|(f^* \circ \delta_t)(Tf- Tg)|}{\|f - g\|^2} &=& \frac{|f^*((S(f(t))- S(g(t)))|}{\|f (t) - g(t) \|^2}\\
		&\leq&\nu_L(S).
		\end{eqnarray*}
		Therefore, $n_L(A(\Omega, X)) \leq n_L(X)$.
		\end{proof}
		
		\begin{Corollary}
		Let $\Omega$ be a Hausdorff space, $X$ be a Banach space, and $A$ be the base algebra of $C_b (\Omega, X)$. Suppose that the set $\rho A$ of strong peak points is a norming subset of $\Omega$. Then $ n_L(X) = n_L(C_b (\Omega, X)) $.
		\end{Corollary}

		For a Banach space $Y$ and the set $\Omega = B_Y$, consider the operator $Q_S$ as in the hypothesis of Theorem \ref{th:lipnumaox_conv}. Even if $f \in A(B_Y, X)$ is holomorphic on the interior $B_Y^\circ$ of $B_Y$, we mention that $Q_S (f)$ may not be holomorphic on $B_Y^\circ$. For this reason, we do not know if the Lipschitz numerical indices of $X$-valued extension of disk algebras coincide with that of $X$. 
		
			\begin{question}
Let $X$ and $Y$ be Banach spaces. Consider the following $X$-valued disk algebra $\mathcal{A}$ which is either  
\begin{align*}
&\mathcal{A}_b (B_Y, X) = \{ f : B_Y \rightarrow X : f \text{ is holomorphic on } B_Y^\circ \text{ and continuous on } B_Y \}, \text{ or} \\
&\mathcal{A}_u (B_Y, X) = \{ f \in \mathcal{A}_b (B_Y, X) : f \text{ is uniformly continuous on } B_Y \}. 
\end{align*}
Is it true that $n_L (X) = n_L (\mathcal{A} )$? 
\end{question}
For classical numerical index in this line, we refer to \cite[Corollary 8]{L}.

		\subsection{Lipschitz numerical index of $A(K, X)$}\label{subsec:A(K,X)}
		 Now let us consider the case $\Omega = K$. Here we use the strong boundary points instead of the strong peak points. A point $t_0 \in K$ is said to be a {\it strong boundary point} for the uniform algebra $A = A(K)$ if for every open subset $U \subset K$ containing $t_0$, there exists $f \in A$ such that $|f(t_0)| = 1$ and $\sup_{K \setminus U}|f(t)| < 1$. 
		
		%A subset $S \subset X$ is said to be a {\it boundary} if for each $f \in A$, there exists an element $t \in S$ such that $|f(t)| = \|f\|_\infty$. It is well-known that the set of strong boundary points is a boundary \cite[Proposition 4.3.3]{D} and coincides with the Choquet boundary $\Gamma_0$ of $A$ \cite[Theorem 4.3.5]{D}.
		
		To study the Lipschitz numerical index of $A(K, X)$, we shall prove a stricter version of a Urysohn-type lemma provided in \cite[Lemma 2.5]{CGK} which was used to study the Bishop-Phelps-Bollob\'as property of Asplund operators with range in a uniform algebra \cite{CGK} and the Daugavet and diameter two properties in $A(K, X)$ \cite{LT}. 
		 First we recall the following auxillary lemma.
		
		\begin{Lemma}\cite[Lemma 2.2]{CGK}\label{lem:uniflim1}
			Let $A \subset C(K)$ be a uniform algebra, $M \subset \mathbb{C}$ and $g:M \rightarrow \mathbb{C}$ a function that is the uniform limit of a sequence of complex polynomials restricted to $M$. For every $f \in A$ with $f(K) \subset M$ the following statements hold true:
			\begin{enumerate}
					\setlength\itemsep{0.3em}
				\item[\textup{(i)}] If $A$ is unital, then $g \circ f \in A$.
				\item[\textup{(ii)}] If $A$ is non-unital, $0 \in M$ and $g(0) = 0$, then $g \circ f \in A$. 
			\end{enumerate}
		
		\end{Lemma}

		Even though the proof of the following lemma is almost identical to \cite[Lemma 2.5]{CGK}, we include the proof for completeness. This lemma will enable us to construct a function that ``peaks" at a given strong boundary point.
		
		\begin{Lemma}\label{lem:urysohnsb}
			Let $K$ be a compact Hausdorff space. If $t_0$ is a strong boundary point for a uniform algebra $A \subset C(K)$, then for every open subset $U \subset K$ containing $t_0$ and $\eps > 0$, there exists $\varphi \in A$ such that $\varphi(t_0) = \|\varphi\|_{\infty} = 1$, $\sup_{K \setminus U}|\varphi(t)| < \eps$ and 
			\begin{equation}\label{eq:stolz}
			|\varphi(t)| + (1- \eps)|1 - \varphi(t)| \leq 1
			\end{equation}
			for every $t \in K$. 
		\end{Lemma} 
		
		\begin{proof}
			For a given $\eps > 0$, we denote the Stolz region by $St_{\eps} = \{z \in \mathbb{C}: |z| + (1 - \eps) |1 - z| \leq 1)\}$. In view of Theorem 14.19 in \cite{R}, we know that there exists a homeomorphism $\psi: \overline{\mathbb{D}} \rightarrow St_{\eps}$ such that
			\begin{enumerate}
					\setlength\itemsep{0.3em}
				\item[\textup{(i)}] $\psi$ restricted to $\mathbb{D}$ is a conformal mapping onto the interior of $St_{\eps}$, 
				\item[\textup{(ii)}] $\psi(0) = 0$, and
				\item[\textup{(iii)}] $\psi(1) = 1$.
			\end{enumerate}
			Let $\delta \in (0, 1)$ such that $\delta \mathbb{D} \subset \psi^{-1}(\eps^2 \mathbb{D})$ and $t_0$ be a strong boundary point of $A$. Then for every open subset $U \subset K$ containing $t_0$, there exists $f_U \in A$ such that $f_U (t_0) = 1$ and $\sup_{K \setminus U}|f_U (t)| < 1$. Since $A$ is closed under multiplication, by taking suitable powers of $f_U$, we may assume that $\sup_{K \setminus U}|f_U (t)| < \delta$.
			
			Now define a function $\varphi = \psi \circ f_U$. Note that $\psi$ is the uniform limit of a sequence of complex polynomials on $\overline{\mathbb{D}}$. Thus, in view of Lemma \ref{lem:uniflim1}, we see that $\varphi \in A$. Furthermore, the image of $\varphi$ is contained in $St_{\eps}$ and so (\ref{eq:stolz}) is satisfied for every $t \in K$. Notice that $\varphi(t_0) = \psi \circ f_U (t_0) = \psi(1) = 1$ and $\varphi(K \setminus U) = \psi \circ f_U (K \setminus U) \subset \psi (\delta \mathbb{D}) \subset \eps \mathbb{D}$. The proof is finished. 
		\end{proof}
		
		With the aid of Lemma \ref{lem:urysohnsb}, we can prove the following density result for the space $A(K, X)$, which is inspired by the corresponding result for the space $A(\Omega, X)$ in \cite{L}.
		
		\begin{Proposition}\label{prop:hanju}
		Let $K$ be a compact Hausdorff space and $X$ be a Banach space. Given $\eps >0$, if a norm-one element $f \in A(K, X)$ satisfies $\| f(t_0) \| > 1 - \frac{\eps}{5}$ for some strong boundary point $t_0$ for $A$, then there exists a norm-one element $g \in A(K,X)$ such that $\| g(t_0) \| = 1$ and $\|f - g \| < \eps$. 
		\end{Proposition} 
		
		\begin{proof}
		Suppose that $f \in A(K, X)$ satisfies $\| f(t_0) \| > 1 - \frac{\eps}{5}$ for some strong boundary point $t_0$ for $A$. Consider $U =\left\{ t \in K : \left\| f(t) - \frac{ f(t_0)}{\|f(t_0)\|} \right\| < \frac{\eps}{5} \right\}$ which is an open neighborhood of $t_0$. By Lemma \ref{lem:urysohnsb}, there exists $\varphi \in A$ such that $\varphi (t_0) = \| \varphi \| = 1$, $\sup_{K \setminus U} |\varphi (t)| < \frac{\eps}{5}$ and 
		\[
		|\varphi(t)| + \left(1- \frac{\eps}{5} \right)|1 - \varphi(t)| \leq 1
		\]
		for every $t \in K$. 
		Define $g \in A(K, X)$ by 
		\[
		g(t) := \frac{ f(t_0)}{\|f(t_0)\|} \varphi (t) + \left(1- \frac{\eps}{5} \right)(1-\varphi (t)) f(t) 
		\]
		for every $t \in K$. Then the same calculation in the proof of \cite[Theorem 4]{L} shows that $\| f- g \| < \eps$. It is clear that $\| g(t_0) \| = 1$. 
		\end{proof} 
		
		Now we are ready to compare the Lipschitz numerical indices of $A(K, X)$ and $X$.

		\begin{Theorem}
			 Let $K$ be a compact Hausdorff space and $X$ be a Banach space. Then $n_L(X) \leq n_L(A(K, X))$.
		\end{Theorem}
		
		\begin{proof} 
		Note that the set of strong boundary points for $A$ is a norming subset of $\Omega$ \cite[Proposition 4.3.3]{D}. We can adapt the proof of Theorem \ref{th:lipnumaox} by replacing $\rho A$ with the set of strong boundary points and Lemma \ref{lem:urysohnsp} with Lemma \ref{lem:urysohnsb}. 
		\end{proof} 
		
For the following result analogous to Theorem \ref{th:lipnumaox_conv}, we omit the detail of its proof since it is identical to the proof of Theorem \ref{th:lipnumaox_conv} if we use Proposition \ref{prop:hanju} instead of \cite[Theorem 4]{L}.

		\begin{Theorem}\label{Thm:CK_inverse}
		 Let $K$ be a compact Hausdorff space, $X$ be a Banach space. If the range of the mapping $Q_S : A(K, X) \rightarrow C (K, X)$ given by 
		\[ 
		Q_S (f) (t) = S ( f(t)) \quad (f \in A(K,X),\, t \in K), 
		\]
		for every $S \in \Lip (X, X)$, is contained in $A(K, X)$, then $ n_L(X) = n_L(A(K, X)) $.
	\end{Theorem}
		
	As a particular case, we have the following result which already appeared in \cite[Theorem 4.5]{WHT}. 
	
	\begin{Corollary}
	 Let $K$ be a compact Hausdorff space, $X$ be a Banach space. Then $ n_L(X) = n_L(C(K, X))$.
	\end{Corollary}

If we instead consider the Banach space $C_w (K, X)$ of all weakly continuous functions from $K$ into $X$ equipped with the supremum norm, the main difficulty in comparing the Lipschitz numerical index of $C_w (K, X)$ with that of $X$ is to construct a suitable function $\Phi : X \rightarrow C_w (K, X)$ by using a Urysohn-type lemma as in the proof of Theorem \ref{th:lipnumaox}. Moreover, we cannot always guarantee that $Q_S (f)$ is an element of $C_w (K, X)$ for $S \in \Lip (X, X)$ and $f \in C_w (K,X)$ where $Q_S$ is the operator given in Theorem \ref{Thm:CK_inverse}. For similar reasons, we do not know if the Lipschitz numerical index of $C_{w^*} (K, X^*)$ coincides with that of $X^*$ (or of $X$), where $C_{w^*} (K, X^*)$ is the Banach space of weak-star continuous functions from $K$ to $X^*$. 
	
	\begin{question}
Let $K$ be a compact Hausdorff space, $X$ be a Banach space. Is it possible to estimate the Lipschitz numerical index of $C_w (K, X)$ or of $C_{w^*} (K, X^*)$? 
\end{question}
 
In the case of the classical numerical index, it is observed in \cite{LMM} that $n (X) = n (C_w (K, X))$ and $n(X) \geq n (C_{w^*} (K, X^*)$. In the same paper, it is also shown that if $X$ is Asplund or $K$ has a dense subset of isolated points, then $n(X^*) \leq n(C_{w^*} (K, X^*))$.

\subsection{Inheritance of the Daugavet property}\label{subsection:DP}
Here we study the inheritance of the Daugavet property from $X$ to the vector-valued function space $A(\Omega,X)$ in this subsection. The tools used in Theorem \ref{th:lipnumaox} and the Lipschitz Daugavet center turn out to be the main ingredients to show this inheritance result. 

\begin{Theorem}
Let $\Omega$ be a Hausdorff space, $X$ be a Banach space, and $A$ be the base algebra of $A(\Omega, X)$. Suppose that the set $\rho A$ of strong peak points is a norming subset of $\Omega$. If $X$ has the Daugavet property, then $A(\Omega,X)$ also has the Daugavet property.
\end{Theorem}

\begin{proof}
Suppose that $X$ has the Daugavet property. Fix any rank-one Lipschitz map $T \in \Lip(A(\Omega,X),A(\Omega,X))$ with $\|T\|_L=1$. Let $S \in \Lip(X,X)$ be defined by $S(x) := T_{t_0}(\Phi(x)) - T_{t_0}(\Phi(0))$, where $T_{t_0}$ and $\Phi$ are given as in the proof of Theorem \ref{th:lipnumaox}. Notice here that $S$ is also of rank-one. Thus, there exist $u,v \in X$ with $u \neq v$ such that
$$
\left\| \frac{u-v}{\|u-v\|} + \frac{Su-Sv}{\|u-v\|} \right\| > 1+ \|S\|_L -\eps > 2-8\eps,
$$
by following the arguments in the proof of Theorem \ref{th:lipnumaox}. Hence, we have that
\begin{align*}
\|\Id_{A(\Omega,X)} + T\|_L &\geq \left\| \frac{\Phi(u)-\Phi(v)}{\|\Phi(u)-\Phi(v)\|} + \frac{T(\Phi(u))-T(\Phi(v))}{\|\Phi(u)-\Phi(v)\|} \right\| \\
&\geq \left\| \frac{u-v}{\|u-v\|} + \frac{Su-Sv}{\|u-v\|} \right\| \\
&\geq 2-8\eps.
\end{align*}
So we have shown that $\Id_{A(\Omega,X)}$ on $A(\Omega,X)$ is the Lipschitz Daugavet center, which implies that $A(\Omega,X)$ has the Daugavet property by \cite[Corollary 3.6]{KMMW} (or by Theorem \ref{Thm:SCD}). 
\end{proof}

An analogous result also holds when we consider the space $A(K,X)$. Since the proof is similar to $A(\Omega, X)$, we omit the detail of the proof.

\begin{Theorem}
Let $K$ be a compact Hausdorff space and $X$ be a Banach space. If $X$ has the Daugavet property, then $A(K,X)$ also has the Daugavet property.
\end{Theorem}

Note that the above result generalizes \cite[Remark 6]{MP} that shows the inheritance of the Daugavet property from $X$ to $C(K,X)$. 

		\section{Lipschitz numercial index on the absolute sum}\label{section:absolute-sum}

In this section, we provide other stability results on the Lipschitz numerical index. The section concerns the space with absolute sums, the K\"othe-Bochner spaces, and Banach spaces with increasing one-complemented subspaces whose union is dense. This allows us to establish the stability results on the $\ell_p$-sum of spaces in the spirit of \cite{MMPR}. 

\subsection{Lipschitz numerical index on the absolute sum}\label{subsection:absolute-sum}

In \cite[Corollary 4.4]{WHT}, they proved that the Lipschitz numerical index is stable under the $c_0$-sums, $\ell_1$-sums and $\ell_\infty$-sums of spaces. These sums can be extended to the name of absolute sums, which we will introduce below briefly.

Let $\Lambda$ be any nonempty index set and $E$ be a linear subspace of $\R^\Lambda$. Recall that according to \cite{MMPR} an \emph{absolute norm} on $E$ is a complete norm $\|\cdot\|_E$ satsfying that
\begin{enumerate}
\setlength\itemsep{0.3em}
\item[\textup{(i)}] given any $(a_\lambda),(b_\lambda) \in \R^\Lambda$ with $|a_\lambda|=|b_\lambda|$ for each $\lambda \in \Lambda$, if $(a_\lambda) \in E$ then $(b_\lambda) \in E$ with $\|(a_\lambda)\|_E = \|(b_\lambda)\|_E$,
\item[\textup{(ii)}] for every $\lambda \in \Lambda$, the characteristic function $\chi_{\{\lambda\}}$ of the singleton $\{\lambda\}$ satisfies that $\chi_{\{\lambda\}} \in E$ with $\|\chi_{\{\lambda\}}\|_E=1$.
\end{enumerate}
An \emph{absolue sum} of a family $\{X_\lambda: \lambda \in \Lambda\}$ is the Banach space defined by
$$
\biggl[ \bigoplus_{\lambda \in \Lambda} X_\lambda \biggr]_E := \left\{ (x_\lambda) : x_\lambda \in X_\lambda \text{ for each } \lambda \in \Lambda, \left( \|x_\lambda\| \right) \in E \right\}.
$$
We write by $E'$ the \emph{K\"othe dual} of $E$ given by
$$
E' := \left\{ (b_\lambda) \in \R^\Lambda : \|(b_\lambda)\|_{E'} := \sup_{(a_\lambda) \in B_E} \sum_{\lambda \in \Lambda} |b_\lambda||a_\lambda| < \infty \right\},
$$
which is a linear subspace of $\R^\Lambda$.

We are now able to prove the Lipschitz version of \cite[Theorem 2.1]{MMPR}. Here, we use a slightly different hypothesis on the set $A$ comparing to the one in \cite[Theorem 2.1]{MMPR}, but it can be easily shown that these are equivalent to each other.

\begin{Theorem}\label{Theorem:absolute-sum}
Let $\Lambda$ be a nonempty set and $E$ be a linear subspace of $\R^\Lambda$ endowed with an absolute norm $\|\cdot\|_E$. Suppose that there is a dense subset $A \subseteq S_E^+$ such that given $(a_\lambda) \in A$, there exists $(b_\lambda) \in S_{E'}$ such that $(b_\lambda)(a_\lambda)=1$, where $S_E^+ := \bigl\{ (s_\lambda) \in E : (s_\lambda) = (|s_\lambda|) \bigr\}$. Then, for an arbitrary family $\{X_\lambda: \lambda \in \Lambda\}$ of Banach spaces, we have
$$
n_L \biggl( \biggl[ \bigoplus_{\lambda \in \Lambda} X_\lambda \biggr]_E \biggr) \leq \inf_{\lambda \in \Lambda} n_L(X_\lambda).
$$
\end{Theorem}

\begin{proof}
We write $X = \Bigl[ \bigoplus_{\lambda \in \Lambda} X_\lambda \Bigr]_E$ and $X' = \Bigl[ \bigoplus_{\lambda \in \Lambda} X_\lambda^* \Bigr]_{E'} \subseteq X^*$. Let $S \in S_{\Lip(X_\kappa,X_\kappa)}$ be given. For a fixed $\kappa \in \Lambda$, define $T \in \Lip(X, X)$ by $T := I_\kappa SP_\kappa$. It is routine to show that $\|T\|=1$. We claim that $\nu_L(T) \leq \nu_L(S)$, which shows that $n_L(X) \leq n_L(X_\kappa)$ since $n_L(X) \leq \nu_L(T) \leq \nu_L(S)$ and $S \in S_{\Lip(X_\kappa,X_\kappa)}$ was arbitrary. Let $\mathcal{A} \subseteq S_X$ be a subset defined by
$$
\mathcal{A} := \{ (z_\lambda) \in X : (\|z_\lambda\|) \in A \}.
$$
Then, $\mathcal{A}$ is dense in $S_X$. Let us define
$$
Q : = \left\{\bigl((x_\lambda),(y_\lambda),(x_\lambda^*)\bigr) \in X \times X \times X' : \frac{(x_\lambda-y_\lambda)}{\|(x_\lambda-y_\lambda)\|} \in \mathcal{A}, (x_\lambda^*) \in D\bigl((x_\lambda - y_\lambda)\bigr)\right\}.
$$
We claim that $\pi(Q)$ is dense in $X \times X$. Indeed, given any $\bigl((x_\lambda),(y_\lambda)\bigr) \in X \times X$ and $\eps>0$, consider
$$
(z_\lambda) := \frac{(x_\lambda-y_\lambda)}{\|(x_\lambda-y_\lambda)\|} \in S_X.
$$
By assumption, there exists $(z_\lambda') \in \mathcal{A}$ so that $\|(z_\lambda'-z_\lambda)\|<\eps \min\{ \|(x_\lambda-y_\lambda)\|^{-1}, 1 \}$. 
Take $(x_\lambda') = (x_\lambda)$ and $(y_\lambda') = (x_\lambda') + \|x_\lambda - y_\lambda \| ( z_\lambda' )$. Then 
$$
\|(x_\lambda'-x_\lambda)\| = 0 <\eps, \quad \|(y_\lambda'-y_\lambda)\|<\eps \quad \text{and} \quad \frac{(x_\lambda'-y_\lambda')}{\|(x_\lambda'-y_\lambda')\|} = (z_\lambda').
$$
Next, choose for each $\lambda \in \Lambda$ a functional $x_\lambda^* \in S_{X_\lambda*}$ such that $x_\lambda^*(z_\lambda') = \|z_\lambda'\|$. By hypothesis there is $(b_\lambda) \in S_{E'}$ such that $(b_\lambda)(\|z_\lambda'\|) = 1$. Thus, we obtain that 
$$
\Bigl\{ \bigl( (x_\lambda'), (y_\lambda'), \|(x_\lambda'-y_\lambda')\| (b_\lambda x_\lambda^*) \bigr) \Bigr\} \in Q
$$
and $\pi (Q)$ is dense in $X \times X$. 
It follows from Lemma \ref{lem:radiusattain} that 
$$
\nu_L(T) = \sup \left\{ \left| \frac{(x_\lambda^*)\bigl(T(x_\lambda)-T(y_\lambda)\bigr)}{\|(x_\lambda-y_\lambda)\|^2} \right| : \bigl( (x_\lambda), (y_\lambda), (x_\lambda^*) \bigr) \in Q \right\}.
$$
Therefore,
\begin{align*}
\left| \frac{(x_\lambda^*)\bigl(T(x_\lambda)-T(y_\lambda)\bigr)}{\|(x_\lambda-y_\lambda)\|^2} \right| = \left| \frac{(x_\lambda^*)(I_\kappa Sx_\kappa-I_\kappa Sy_\kappa)}{\|(x_\lambda-y_\lambda)\|^2} \right| \leq \left| \frac{\frac{\|x_\kappa-y_\kappa\|}{\|x_\kappa^*\|} x_\kappa^* (Sx_\kappa - Sy_\kappa)}{\|(x_\lambda-y_\lambda)\|^2} \right|,
\end{align*}
where the last inequality follows from $\|x_\kappa^*\|\|x_\kappa-y_\kappa\| \leq \|(x_\lambda-y_\lambda)\|^2$. Notice also from the fact $x_\lambda^*(x_\lambda-y_\lambda) = \|x_\lambda^*\|\|x_\lambda-y_\lambda\|$ for all $\lambda \in \Lambda$ that
$$
\frac{\|x_\kappa-y_\kappa\|}{\|x_\kappa^*\|} x_\kappa^* \in D(x_\kappa-y_\kappa),
$$
and hence we can deduce that $\nu_L(T) \leq \nu_L(S)$.
\end{proof}

If the space $E$ is order-continuous, the hypothesis on the denseness automatically follows from the equality $E'=E^*$. Also, as we can always find a norm attaining functional in $E^*$ close to the given any functional by the Bishop-Phelps theorem, we obtain the following conseqeunces, a Lipschitz version of \cite[Corollary 2.2]{MMPR}.

\begin{Corollary}\label{Corollary:order-continuous-sum}
Let $\Lambda$ be a nonempty set and $E$ be a linear subspace of $\R^\Lambda$ endowed with an absolute norm $\|\cdot\|_E$ which is order continuous. Then, for an arbitrary family $\{X_\lambda: \lambda \in \Lambda\}$ of Banach spaces, we have
$$
n_L \biggl( \biggl[ \bigoplus_{\lambda \in \Lambda} X_\lambda \biggr]_E \biggr) \leq \inf_{\lambda \in \Lambda} n_L(X_\lambda) \quad \text{and} \quad n_L \biggl( \biggl[ \bigoplus_{\lambda \in \Lambda} X_\lambda \biggr]_{E'} \biggr) \leq \inf_{\lambda \in \Lambda} n_L(X_\lambda).
$$
\end{Corollary}

The following cases are more visible examples, which can be derived from the fact that $E$ is order-continuous.

\begin{Corollary}
\textup{(a)} Let $E$ be the space $\R^m$ endowed with an absolute norm $\|\cdot\|_E$ and let $X_1, \ldots, X_m$ be Banach spaces. Then,
$$
n_L\bigl([X_1 \oplus \cdots \oplus X_m]_E\bigr) \leq \min \{ n_L(X_1), \ldots, n_L(X_m) \}.
$$
\textup{(b)} Let $E$ be a Banach space with an 1-unconditional basis and let $\{X_j : j \in \N\}$ be a sequence of Banach spaces. Then,
$$
n_L \biggl( \biggl[ \bigoplus_{j \in \N} X_j \biggr]_E \biggr) \leq \inf_{j \in \N} n_L(X_j).
$$
\end{Corollary}

In \cite[Theorem 2.5]{MMPR}, the authors show that the converse inequality of Theorem \ref{Theorem:absolute-sum} with constraints also holds. However, concerning the case of Lipschitz maps, the choice of the set $A \subseteq S_E$ such that $\overline{\conv}(A) = B_E$ is not enough to control the behavior of $T$ since it lacks the linearity.

\begin{question}
In the above setting \textup{(}possibly with more conditions\textup{)}, does the inequality
$$
n_L \biggl( \biggl[ \bigoplus_{\lambda \in \Lambda} X_\lambda \biggr]_E \biggr) \geq \inf_{\lambda \in \Lambda} n_L(X_\lambda).
$$
hold in general?
\end{question}

\subsection{Lipschitz numerical index on the K\"othe spaces}\label{subsection:Kothe}

Given a complete $\sigma$-finite measure space $(\Omega,\Sigma,\mu)$, let $L_0(\mu)$ be the space of all $\Sigma$-measurable locally integrable functions $f:\Omega \to \R$. A subspace $E$ of $L_0(\mu)$, which is also known as the \emph{K\"othe function space}, is the space endowed with the norm $\|\cdot\|_E$ such that
\begin{enumerate}
\setlength\itemsep{0.3em}
\item[\textup{(i)}] $f \in E$ and $\|f\|_E \leq \|g\|_E$ whenever $g \in E$ and $|f| \leq |g|$ almost everywhere on $\Omega$,
\item[\textup{(ii)}] for any $A \in \Sigma$ with $0<\mu(A)<\infty$, the characterstic function $\chi_A:\Omega \to \R$ lies in $E$.
\end{enumerate}
The \emph{K\"othe dual} $E'$ of $E$ is known to be the function space defined by
$$
E' := \left\{ g \in L_0(\mu): \|g\|_{E'} := \sup_{f \in B_E} \int_\Omega |fg| d\mu < \infty\right\},
$$
which can also be viewed as a K\"othe space on $(\Omega,\Sigma,\mu)$. Given $g \in E'$, the mapping
$$
f \mapsto \int_\Omega fg d\mu \qquad \text{for } f \in E
$$
becomes an element of $E^*$, which shows that $E' \subseteq E^*$ isometrically. A \emph{strongly measurable} function $f: \Omega \to X$ is a function which can be approximated by simple functions almost everywhere on $\Omega$. We write the \emph{K\"othe-Bochner function space} by $E(X)$ to denote the space of all strongly measurable functions $f: \Omega \to X$ such that $(t \mapsto \|f(t)\|_X) \in E$ equipped with the norm $\|\cdot\|_{E(X)} = \bigl\| t \mapsto \|\cdot(t)\|_X \bigr\|_E$. The dual space of $E(X)$ can be identified with the space $E'(X^*,w^*)$ which consists of all $w^*$-scalarly measurable functions $\Phi:\Omega \to X^*$ such that $\|\Phi(\cdot)\|_{X^*} \in E'$ which an action
$$
\langle \Phi, f \rangle = \int_\Omega \langle \Phi(t), f(t)\rangle d \mu(t) \qquad \text{for } f \in E(X)
$$
(see \cite[Theorem 3.2.4]{Lin}). For the detailed information on the K\"othe spaces, we refer to \cite{LT2} for its background. In \cite{MMPR}, the one-sided inequality between the numerical indices of K\"othe-Bochner space $E(X)$ and the underlying space $X$ has been shown. Now we provide the Lipschitz version of this result. 
%Notice that the lack of linearity on Lipschitz maps creates significant changes in the proof.

\begin{Theorem}\label{Theorem:Kothe-space}
Let $X$ be a Banach space, $(\Omega,\Sigma,\mu)$ be a complete $\sigma$-finite measure space and $E$ be a K\"othe space on $(\Omega,\Sigma,\mu)$. Suppose that there exists a dense subset $\mathcal{A} \subseteq S_{E(X)}$ such that for every $f \in \mathcal{A}$ there is $\Phi_f \in E'(X^*,w^*)$ satisfying $\||\Phi_f|\|_{E'} = \langle \Phi_f, f \rangle =1$. Then, we have
$$
n_L(E(X)) \leq n_L(X).
$$
\end{Theorem}

\begin{proof}
Let $S \in \Lip(X,X)$ be such that $\|S\|_L=1$. Define $T: E(X) \to E(X)$ by
$$
(Tf)(t) := S(f(t)) \qquad \text{for } t \in \Omega \text{ and } f \in E(X).
$$
We first claim that $Tf$ is strongly measurable for any $f \in E(X)$. To see this, take $h = \sum_{i=1}^n x_i \chi_{A_i}$ such that $\|h(t)-f(t)\|_X < \eps$ for almost every $t \in \Omega$. Then, it is clear that $(Th)(t) = \sum_{i=1}^n S(x_i) \chi_{A_i}(t)$ and
$$
\|(Tf)(t) - (Th)(t)\| = \|S(f(t)) - S(h(t))\| \leq \|f(t)-h(t)\|_X < \eps.
$$
Next, observe that
$$
\|(Tf)(t) - (Tg)(t)\|_X \leq \|f(t)-g(t)\|_X = |f-g|(t) \quad \text{for } t \in \Omega.
$$
This implies that $Tf \in E(X)$ for all $f \in E(X)$ and
$$
\bigl\| t \mapsto \|(Tf)(t)-(Tg)(t)\|_X \bigr\|_E \leq \||f-g|\|_E = \|f-g\|_{E(X)},
$$
thus $\|Tf-Tg\|_{E(X)} \leq \|f-g\|_{E(X)}$. To deduce that $\|T\|_L =1$, we fix any $A \in \Sigma$ with $0<\mu(A)<\infty$. Given any $x,y \in X$ with $x \neq y$, define $f,g \in E(X)$ by $f := x \chi_A$ and $g := y \chi_A$. Then, we have
\begin{align*}
\|Tf-Tg\|_{E(X)} &= \bigl\| t \mapsto \|S(x \chi_A(t)) - S(y \chi_A(t)) \|_X \bigr\|_E \\
&= \bigl\| t \mapsto \|Sx-Sy\|_X \chi_A(t) \bigr\|_E = \|Sx-Sy\|_X \|\chi_A\|_E
\end{align*}
and
$$
\|f-g\|_{E(X)} = \|x-y\|_X \|\chi_A\|_E.
$$
Thus, from the equality
$$
\frac{\|Tf-Tg\|_{E(X)}}{\|f-g\|_{E(X)}} = \frac{\|t \mapsto \|Sx-Sy\|_X\|_E}{\|t \mapsto \|x-y\|_X \|_E} = \frac{\|Sx-Sy\|}{\|x-y\|},
$$
we obtain that $\|T\|_L=1$.

On the other hand, we consider the set
$$
Q := \left\{ (f,g,\Psi_h) : f,g \in E(X) \text{ with } f \neq g, \, h :=\frac{f-g}{\|f-g\|} \in \mathcal{A} \text{ and } \langle \Psi_h, h \rangle = \frac{1}{\|f-g\|} \right\}
$$
in $\Pi_{E(X)}$, where $\Psi_h := \|f-g\| \Phi_h$. By considering suitable perturbations on $f$ and $g$ (see the proof of Theorem \ref{Theorem:absolute-sum}), we can show that $\pi(Q)$ is dense in $E(X) \times E(X)$. Given $(f,g,\Psi_h) \in Q$, it follows that
\begin{align*}
&\left| \frac{\|f-g\| \langle \Psi_h, Tf - Tg \rangle}{\|f-g\|^2} \right| \\
&\quad= \frac{1}{\|f-g\|} \bigg| \int_\Omega \left\langle \|f(t)-g(t)\| \frac{\Phi_h(t)}{\|\Phi_h(t)\|}, \frac{S(f(t))-S(g(t))}{\|f(t)-g(t)\|^2} \right\rangle \|f(t)-g(t)\| \|\Phi_h(t)\| d \mu(t) \bigg| \\
&\quad\leq \frac{1}{\|f-g\|} \int_\Omega \nu_L(S) \|f(t)-g(t)\| \|\Phi_h(t)\| d\mu(t) \\
&\quad= \nu_L(S).
\end{align*}
Here, we used the facts from the equation
\begin{align*}
\langle \Psi_h, f-g \rangle &= \|f-g\| \int \left\langle \Phi_h(t) , f(t)-g(t) \right\rangle d \mu(t) \\
&\leq \|f-g\| \int_\Omega \|\Phi_h(t)\| \|f(t)-g(t)\| d\mu(t) \\
&\leq \|f-g\| \bigl\| t \mapsto \|\Phi_h(t)\| \bigr\|_{E'} \bigl\| t \mapsto \|f(t)-g(t)\| \bigr\|_E \\
&\leq \|f-g\|^2
\end{align*}
which also implies that 
\[
\left( f(t), g(t), \|f(t)-g(t)\| \dfrac{\Phi_h(t)}{\|\Phi_h(t)\|} \right) \in \Pi_X. 
\]
Consequently, by Lemma \ref{lem:radiusattain}, we have that $\nu_L (T) \leq \nu_L (S)$; hence $n_L(E(X)) \leq n_L(X)$.
\end{proof}

Thanks to \cite[Corollary 4.2]{MMPR}, we can also obtain the following consequence since an order-continuous K\"othe space satisfies the condition given in Theorem \ref{Theorem:Kothe-space}.

\begin{corollary}\label{cor:ordercont}
Let $X$ be a Banach space, $(\Omega,\Sigma,\mu)$ be a complete $\sigma$-finite measure space and $E$ be an order-continuous K\"othe space on $(\Omega,\Sigma,\mu)$. Then, we have
$$
n_L(E(X)) \leq n_L(X).
$$
\end{corollary}

Since the space $L_p$ is an order-continuous K\"othe space for $1 \leq p <\infty$, we obtain the following consequence that includes the well-known result on the Lipschitz numerical index of $L_1(\mu, X)$ for a positive measure $\mu$ in \cite{WHT}. 

\begin{corollary}\label{corollary:L_p(mu,X)}
Let $X$ be a Banach space, $(\Omega,\Sigma,\mu)$ be a complete $\sigma$-finite measure space. Then for $1 \leq p <\infty$, we have
$$
n_L(L_p(\mu,X)) \leq n_L(X).
$$
\end{corollary}

\subsection{Lipschitz numerical index on increasing one-complemented spaces}\label{subsection:increasing}

Even though we could only show in general that an upper bound of the Lipschitz numerical index of an absolute sum of Banach spaces in the previous section, a lower bound of such a space can be obtained under a certain condition on the each component of the sum. Inspired by \cite[Theorem 5.1]{MMPR}, we prove the following result. 

\begin{proposition}\label{proposition:direct-set-union}
Let $X$ be a Banach space, $I$ be a direct set and $\{X_i:i \in I\}$ be an increasing family of one-complemented closed subspaces of $X$ whose union is dense in $X$. Then, we have
$$
n_L(X) \geq \limsup_i n_L(X_i).
$$
\end{proposition}

\begin{proof}
Let $\eps>0$ be given. We choose $T \in S_{\Lip(X,X)}$ such that $\nu_L(T) \leq (1+\eps)n_L(X)$. For each $i \in I$, let $P_i: X \to X_i$ and $E_i: X_i \to X$ be the canonical projection and the natural injection, respectively. Define each operator $S \in \Lip(X_i,X_i)$ by $S_i := P_i TE_i$. First, we claim that $\nu_L(S_i) \leq \nu_L(T)$. Given $(x,y,x^*) \in \Pi_{X_i}$, observe that 
\begin{align*}
\frac{|x^*(S_ix-S_iy)|}{\|x-y\|^2} &= \frac{|x^*(P_iTE_ix-P_iTE_iy)|}{\|x-y\|^2} = \frac{|(P_i^*x^*)(TE_ix-TE_iy)|}{\|E_ix-E_iy\|^2} \leq \nu_L(T),
\end{align*}
since $(E_ix,E_iy,P_i^*x^*) \in \Pi_X$. This proves the claim. 

Next, we claim that $\lim_i \|S_i\|_L=1$. It is clear from the construction that $\|S_i\|_L \leq 1$ for each $i \in I$. Given $0<\delta<1$, choose $u,v \in X$ such that
$$
\frac{\|Tu-Tv\|}{\|u-v\|}>1-\frac{\delta}{8}.
$$
Fix a constant $\eta := \min \left\{\delta, \|u-v\| \delta \right\}>0$. Since $\{X_i\}_{i \in I}$ is increasing and its union is dense in $X$, we may find an index $i_0$ such that if $i \geq i_0$, then there exist points $x_1,x_2,y_1,y_2 \in X_i$ with $x_1 \neq x_2$ and $y_1 \neq y_2$ such that
$$
\|Tu-E_ix_1\|<\frac{\eta}{8}, \quad \|Tv-E_ix_2\|<\frac{\eta}{8}, \quad \|u-E_iy_1\|<\frac{\eta}{8} \quad \text{and} \quad \|v-E_iy_2\|<\frac{\eta}{8}.
$$
Then, from the immediate inequalities
$$
\bigl|\|x_1 - x_2\| - \|Tu-Tv\| \bigr| \leq \frac{\eta}{4} \quad \text{and} \quad \bigl|\|y_1 - y_2\| - \|u-v\| \bigr| \leq \frac{\eta}{4},
$$
we have
\begin{align*}
\|S_iy_1-S_iy_2\| &= \|P_iTE_i y_1 - P_iTE_iy_2\| \\
&\geq \|P_iE_ix_1 - P_iE_ix_2\| - \|P_iE_ix_i - P_iTu\| - \|P_iE_ix_2 - P_iTv\| \\
&\phantom{--}-\|P_iTu - P_iTE_iy_1\| - \|P_iTv - P_iTE_iy_2\| \\
&\geq \|x_1-x_2\| - \frac{1}{2}\eta \\
&\geq \|Tu-Tv\| -\frac{3}{4}\eta.
\end{align*}
This leads to
$$
\|S_iy_1 - S_iy_2 \| \geq \left( 1 - \frac{\delta}{8} \right) \|u-v\| - \frac{3}{4}\eta \geq \left(1-\frac{7}{8}\delta\right) \left(1+\frac{\delta}{4}\right)^{-1} \left\|y_1-y_2\right\|
$$
and by choosing $\delta$ converging to 0, we can deduce that $\lim_i \|S_i\|_L=1$.

Thus, for every $i \in I$, we get
$$
(1+\eps) n_L(X) \geq \nu_L(T) \geq \nu_L(S_i) \geq n_L(X_i) \|S_i\|_L;
$$
hence $n_L(X) \geq \limsup_i n_L(X_i)$.
\end{proof}

As a direct corollary of Proposition \ref{proposition:direct-set-union}, we can obtain the following.

\begin{corollary}
Let $X$ be a Banach space with a monotone basis $\{e_i\}$, and let $X_i := \operatorname{span}\{e_j : 1 \leq j \leq i\}$ for each $i \in \N$. Then, we have
$$
n_L(X) \geq \limsup_i n_L(X_i).
$$
\end{corollary}

\subsection{Stability on the $\ell_p$-sum}\label{subsection:ell_p-sum}

As a consequence of what we have studied so far, we can show the following relationship among the Lipschitz numerical indices of vector-valued $L_p$-spaces. Here $\ell_p^n(X)$ stands for the space $\Bigl[ \bigoplus_{i=1}^n X \Bigr]_{\ell_p}$.

\begin{proposition}
Let $X$ be a Banach space, $\mu$ be a positive measure and $1<p<\infty$ be given. Then, we have the following:
\begin{enumerate}
\setlength\itemsep{0.3em}
\item[\textup{(a)}] $\displaystyle n_L(\ell_p(X)) = \lim_n n_L(\ell_p^n(X))$.
\item[\textup{(b)}] $n_L(L_p(\mu,X)) = n_L(\ell_p(X))$ for infinite-dimensional $L_p(\mu)$.
\end{enumerate}
\end{proposition}

\begin{proof}
(a): Since $\ell_p$-sum is an absolute sum, $n_L(\ell_p(X)) \leq n_L(\ell_p^n(X))$ for every $n \in \N$ by Corollary \ref{Corollary:order-continuous-sum}. Consider a sequence of space $X_n := E_n (\ell_p^n(X)) \subseteq \ell_p(X)$ where $E_n:\ell_p^n(X) \to \ell_p(X)$ is the natural injection to first $n$ coordinates. Then, by Proposition \ref{proposition:direct-set-union}, we have that $n_L(\ell_p(X)) \geq \limsup_n n_L(X_n) = \lim_n n_L(\ell_p^n(X))$ since $n_L(\ell_p^n(X))$ is non-increasing.

(b): By \cite[Lemma 6.2]{MMPR}, we may write $L_p(\mu,X) = L_p(\tau,X) \oplus_p Z$ for some $\sigma$-finite measure $\tau$ and a Banach space $Z$, where $L_p(\tau,X) = L_p(\tau,\ell_p(X))$. Thus combining Corollaries \ref{Corollary:order-continuous-sum} with  \ref{corollary:L_p(mu,X)} gives the inequality
$$
n_L(L_p(\mu,X)) \leq n_L(L_p(\tau,X)) = n_L(L_p(\tau,\ell_p(X)) \leq n_L(\ell_p(X)).
$$
For the converse, we consider a directed set $I$ and a family of Banach spaces $\{X_i\}$ in the same way constructed in the proof of \cite[Proposition 6.1.(b)]{MMPR}. It then follows that
$$
n_L(L_p(\mu,X)) \geq \limsup_i n_L(X_i) \geq \inf_n n_L(\ell_p^n(X)) = n_L(\ell_p(X)).
$$
\end{proof}


\begin{thebibliography}{widestlabel}

\bibitem{AcoAguPay-Carolinae} 
\textsc{M. D. Acosta, F. J. Aguirre, and R. Pay\'a}, \emph{A space by W. Gowers and new results on norm and numerical radius attaining operators}, {Acta Univ. Carolin. Math. Phys.} \textbf{33} (1992), 5--14.

\bibitem{AK} 
\textsc{M. Acosta and S. G. Kim}, \emph{Denseness of holomorphic functions attaining their numerical radii}, Israel J. Math. \textbf{161} (2007), 373--386.

\bibitem{AKMMS} 
\textsc{A. Avil\'es, V. Kadets, M. Mart\'in, J. Mer\'i, and V. Shepelska}, \emph{Slicely countably determined Banach spaces}, Trans. Amer. Math. Soc. \textbf{362} (2010), 4871--4900.

\bibitem{BL} 
\textsc{Y. Benyamini and J. Lindenstrauss}, \emph{Geometric nonlinear functional analysis}, Vol.1, Amer. Math. Soc. Colloq. Publ. 48, Amer. Math. Soc. 2000. 

\bibitem{BD2}
\textsc{F. F. Bonsall, and J. Duncan}, \emph{Numerical Ranges II}, London Math. Soc. Lecture Note Series \textbf{10}, Cambridge University Press, 1973.

\bibitem{BK} 
\textsc{T. V. Bosenko and V. Kadets}, \emph{Daugavet centers}, Zh. Mat. Fiz. Anal. Geom. \textbf{6} (2010), 3--20.

\bibitem{BKMW}
\textsc{K. Boyko, V. Kadets, M. Mart\'in, and D. Werner}, \emph{Numerical index of Banach spaces and duality}, Math. Proc. Cambridge. Philos. Soc. \textbf{142}~(2007),  93--102.

\bibitem{CapMM} 
\textsc{A. Capel, M. Mart\'in, and J. Mer\'i}, \emph{Numerical radius attaining compact linear operators}, {J. Math. Anal. Appl.} \textbf{445} (2017), 1258--1266.

\bibitem{C2} 
\textsc{C. S. Cardassi}, \emph{Numerical radius-attaining operators on $C(K)$}, {Proc. Amer. Math. Soc.} {\bf 95} (1985), 537--543.

\bibitem{CGK}
\textsc{B. Cascales, A. Guirao, and V. Kadets}, \emph{A Bishop-Phelps-Bollob\'as type theorem for uniform algebras}, Adv. Math. \textbf{240} (2013),  370--382.

\bibitem{CGKM}
\textsc{Y. S. Choi, D. Garc\'ia, S. G. Kim, and M. Maestre}, \emph{The polynomial numerical index of a Banach space}, Proc. Edin. Math. Soc. \textbf{49}~(2006),  39--52. 

\bibitem{CGMM} 
\textsc{Y. S. Choi, D. Garc\'ia, M. Maestre, and M. Mart\'in}, \emph{Polynomial numerical index for some complex vector-valued function spaces}, Q. J. Math. \textbf{59} (2008), 455--474.

\bibitem{CK} 
\textsc{Y. S. Choi, and S. G. Kim}, \emph{Norm or numerical radius attaining multilinear mappings and polynomials}, {J. London Math. Soc.} \textbf{54} (1996), 135--147.

%\bibitem{DS}
%N. Dunford and J. Schwartz, \emph{Linear operators, Part I: General theory}, Interscience, New York, 1957.

\bibitem{D}
\textsc{H. G. Dales}, \emph{Banach algebras and automatic continuity}, vol.24, London Mathematical Society Monographs. New Series, the Clarendon Press, Oxford University Press, New York, NY, USA, 2000.

\bibitem{Dau}
\textsc{I. K. Daugavet}, \emph{A property of completely continuous operators in the space $C$}, Uspekhi Mat. Nauk. \textbf{18}:5(113) (1963), 157--158 (Russian).

\bibitem{DMPW}
\textsc{J. Duncan, C. M. McGregor, J. D. Pryce, and A. J. White}, \emph{The numerical index of a normed space}, J. London. Math. Soc. \textbf{2}~(1970),  481--488. 

\bibitem{FGKLM} 
\textsc{J. Falc\'o, D. Garc\'ia, S. K. Kim, H. J. Lee, and M. Maestre}, \emph{Polarization constant for the numerical radius}, Mediterr. J. Math. \textbf{17} (2020). 

\bibitem{FMP}
\textsc{C. Finet, M. Mart\'in, and R.Pa\'ya}, \emph{Numerical index and renorming}, Proc. Amer. Math. Soc. \textbf{131}, 2003,  871--877.

%\bibitem{GK}
%G. Godefroy and N. Kalton, \emph{Lipschitz-free Banach spaces}, Stud. Math. \textbf{159}~(2003), no. 1,  121--141. 

\bibitem{GGMMM}
\textsc{D. Garc\'ia, B. C. Grecu, M. Maestre, M. Mart\'in, and J. Mer\'i}, \emph{Polynomial numerical indices of $C(K)$ and $L_1(\mu)$}, Proc. Amer. Math. Soc. \textbf{142}(4) (2014), 1229--1235.

\bibitem{H}
\textsc{L. A. Harris}, \emph{The numerical range of holomorphic functions in Banach spaces}, American Journal of Mathematics, vol. \textbf{93} (1971), 1005--1019.

\bibitem{I} 
\textsc{T. Ivashyna}, \emph{Daugavet centers are separably determined}, Mat. Stud. \textbf{40} (2013), 66--70. 

\bibitem{KMMW}
\textsc{V. Kadets, M. Mart\'in, J. Mer\'i, and D. Werner}, \emph{Lipschitz slices and the Daugavet equation for Lipschitz operators}, Proc. Amer. Math. Soc. \textbf{143} (2015),  5281--5292. 

\bibitem{KMMW2}
\textsc{V. Kadets, M. Mart\'in, J. Mer\'i, and D. Werner}, \emph{Lushness, numerical Index 1 and the Daugavet property in rearrangement invariant spaces}, Canad. J. Math. \textbf{65}~(2013), 331--348.

\bibitem{KMMP}
\textsc{V. Kadets, M. Mart\'in, J. Mer\'i, and A. P\'erez}, \emph{Spear operators between Banach spaces}, Lecture Notes in Math. 2205, Springer, 2018.

\bibitem{KMMPQ} 
\textsc{V. Kadets, M. Mart\'in, J. Mer\'i, A. P\'erez, and A. Quero}, \emph{On the numerical index with respect to an operator}, Dissertationes Mathematicae, \textbf{547}~(2020), 1--58. 

\bibitem{KMP} 
\textsc{V. Kadets, M. Mart\'in, and R. Pay\'a}, \emph{Recent progress and open questions on the numerical index of Banach spaces}, Rev. R. Acad. Cien. Serie A. Mat. \textbf{2000} (2006), 155--182.

\bibitem{KMS}
\textsc{V. Kadets, M. Mart\'in, and M. Soloviova}, \emph{Norm-attaining Lipschitz functionals}, Banach J. Math. Anal. \textbf{10}, (2016), 621--637.
	
\bibitem{KL}
\textsc{S. K. Kim and H. J. Lee}, \emph{A Urysohn-type theorem and the Bishop-Phelps-Bollob\'as theorem for holomorphic Functions}, J. Math. Anal. and Appl. \textbf{480} (2019), 123393.

\bibitem{L}
\textsc{H. J. Lee}, \emph{Generalized numerical index of function algebras}, Journal of Function Spaces, Article ID 9080867 (2019), 1--6.

\bibitem{LMM}
\textsc{G. L\'opez, M. Mart\'in, and J. Mer\'i}, \emph{Numerical index of Banach spaces of weakly or weakly-star continuous functions}, Rocky Mountain J. Math. \textbf{38} (2008), 213--223. 

\bibitem{LT}
\textsc{H. J. Lee and H. J. Tag}, \emph{Diameter two properties in some vector-valued function spaces}, Rev. R. Acad. Cienc. Exactas F\'is. Nat. Ser. A Mat. \textbf{116}~(2022), Article No. 17. 

\bibitem{Lin}
\textsc{P. K. Lin}, \emph{K\"othe-Bochner function spaces}, Birk\"auser, 2012. 

\bibitem{LT2}
\textsc{J. Lindenstrauss and L. Tzafriri}, \emph{Classical Banach spaces II}, Springer-Verlag, Berlin-Heidelberg-New York, 1979.

\bibitem{Lo}
\textsc{G. Ya. Lozanovskii}, \emph{On almost integral operators in KB-spaces}, Vestnik Leningrad Univ. Mat. Mekh. Astr, \textbf{21}~(1966), 35--44 (Russian).

\bibitem{MMPR}
\textsc{M. Mart\'in, J. Mer\'i, M. Popov, and B. Randrianantoanina}, \emph{Numerical index of absolute sums of Banach spaces}, J. Math. Anal. and Appl, \textbf{375} (2011),  207--222.

\bibitem{MP}
\textsc{M. Mart\'in and R. Pay\'a}, \emph{Numerical index of vector-valued function spaces}, Studia Math. \textbf{142} (3) (2000), 269--280.

\bibitem{Paya1} 
\textsc{R. Pay\'a}, \emph{A counterexample on numerical radius attaining operators}, {Israel J. Math.} \textbf{79} (1992), 83--101.

\bibitem{RA} 
\textsc{R. A. McGuigan}, \emph{On the connectedness of isomorphism classes}, Manuscripta Math. \textbf{3} (1970), 1--5. 

\bibitem{R}
\textsc{W. Rudin}, \emph{Real and complex analysis}, McGraw-Hill, 1987.

%\bibitem{S}
%E. Santos, \emph{Polynomial Daugavet Centers}, Q. J. Math., \textbf{71}~(2020), Issue 4, pp 1237--1251. 

\bibitem{Sims} 
\textsc{B. Sims}, \emph{On numerical range and its application to Banach algebras}, PhD dissertation, University of Newcastle, Australia, 1972.

\bibitem{W1}
\textsc{R. Wang}, \emph{The numerical index of Lipschitz operators on Banach spaces}, Studia Math. \textbf{209} (1) (2012), 43--51.

\bibitem{WHT}
\textsc{R. Wang, X. Huang, and D. Tan}, \emph{On the numerical radius of Lipschitz operators in Banach spaces}, J. Math. Anal. Appl, \textbf{411} (2014), 1--18.

\bibitem{W2}
\textsc{N. Weaver}, \emph{Lipschitz algebras}, Second edition. World Scientific Publishing Co. Pte. Ltd. Hackensack, NJ, 2018.

\bibitem{W}
\textsc{D. Werner}, \emph{Recent progress on the Daugavet property}, Irish Math. Soc. Bull. \textbf{46}~(2001), 77--97.

\bibitem{W2}
\textsc{D. Werner}, \emph{The Daugavet equation for operators not fixing a copy of $C[0,1]$}, J. Operator Theory, \textbf{39}~(1998), 89--98.

\bibitem{Wo}
\textsc{P. Wojtaszczyk}, \emph{Some remarks on the Daugavet equation}, Proc. Amer. Math. Soc. \textbf{115}~(1992), 1047--1052.

\bibitem{Z}
\textsc{E. H. Zarantonello}, \emph{The closure of the numerical range contains the spectrum}, Pacific. J. Math. \textbf{22}~(1967), 575--595.

	 
\end{thebibliography}
\end{document}